\documentclass[11pt,a4paper]{amsart}
 
\usepackage{amsmath,amssymb}
\usepackage{amsthm} 
\usepackage[utf8]{inputenc} 
\usepackage{a4wide}
\usepackage{graphicx}
\usepackage{enumitem}
\usepackage{xcolor}

\usepackage[style=numeric, 
  sorting=nyt,
  giveninits=true,
  defernumbers=true,
  backend=biber]{biblatex}

\DeclareFieldFormat{titlecase}{#1}

\newcommand{\UP}{\blacktriangle}
\newcommand{\DOWN}{\blacktriangledown}
\newcommand{\Up}{\vartriangle}
\newcommand{\Down}{\triangledown}

\theoremstyle{plain}
\newtheorem{theorem}{Theorem}[section]
\newtheorem{proposition}[theorem]{Proposition}
\newtheorem{lemma}[theorem]{Lemma}
\newtheorem{corollary}[theorem]{Corollary}

\theoremstyle{definition}
\newtheorem{definition}[theorem]{Definition}
\newtheorem{example}[theorem]{Example}

\theoremstyle{remark}
\newtheorem{remark}[theorem]{Remark}

\setlist[itemize]{label=$\circ$, topsep=2pt}
\setlist[enumerate]{topsep=2pt, itemsep=1pt}
\setlist[description]{
  font=\normalfont\itshape,
  topsep=2pt, itemsep=1pt
}

\numberwithin{equation}{section}
\begin{document}

\date{\today}

\title[Characterising some Stone and Kleene algebras of rough sets]%
{Characterising regular double Stone and pseudocomplemented Kleene algebras in completions of rough sets induced by reflexive relations}

\author{Jouni J\"{a}rvinen}
\address[J.~J{\"a}rvinen]{Software Engineering, LUT School of Engineering Science, Mukkulankatu~19, 15210 Lahti, Finland}
\email{jouni.jarvinen@lut.fi}

\author{S\'{a}ndor Radeleczki}
\address[S.~Radeleczki]{Institute of Mathematics, University of Miskolc, 3515 Miskolc-Egyetemv\'{a}ros, Hungary}
\email{sandor.radeleczki@uni-miskolc.hu}


\begin{abstract}
We consider Kleene and Stone algebras defined on the completion DM(RS) of the ordered set of rough sets induced by a reflexive relation. We focus on the cases where the completion forms a spatial and completely distributive lattice. We derive the conditions under which DM(RS) is a regular pseudocomplemented Kleene algebra and a completely distributive double Stone algebra. Finally, we describe the reflexive relations for which DM(RS) forms a regular double Stone algebra, which is the same structure as in the case of equivalences. Our results generalise earlier findings on algebras of rough sets induced by equivalences, quasiorders, and tolerance relations.
\end{abstract}

\maketitle

\hfill%
Dedicated to the memory of E.~Tam\'{a}s Schmidt

\section{Introduction} \label{sec:intro}
Kleene and Stone algebras are essential in the study of non-classical 
logics. They generalise Boolean algebras by relaxing certain 
constraints to handle negation and ``intermediate'' truth values. In 
this paper, we show how pseudocomplemented Kleene algebras, Stone algebras, and regular double Stone algebras can be defined in terms 
of rough sets induced by reflexive relations.

Rough Set Theory offers a powerful foundation for various contemporary soft computing methods.
Rough sets were introduced by Z.~Pawlak in \cite{Pawl82}. 
In rough set theory, our knowledge about the elements of a 
universe $U$ is given in terms of an equivalence relation $E$. 
Two elements $x,y \in U$ are $E$-related if they are indistinguishable with respect to the available information.

The literature contains numerous studies in which 
information about objects is given in terms of other types 
of relations generalising equivalences. For instance, rough approximations defined by an arbitrary  binary relation were considered as early as \cite{Yao96}. 
If $R$ is a given binary relation on $U$, then for any 
subset $X \subseteq U$, the lower
approximation of $X$ is defined as
\[ X^\DOWN := \{x \in U \mid R(x) \subseteq X \} \]
and the upper approximation of X is
\[ X^\UP := \{x \in U \mid R(x) \cap X \neq \emptyset\} \]
where $R(x) := \{y \in U \mid (x,y) \in R \}$.
The \emph{rough set} of $X \subseteq U$ is the pair 
$(X^\DOWN,X^\UP)$. The set of all rough sets is denoted
by $\mathrm{RS}$. The set $\mathrm{RS}$ is ordered by the coordinatewise order:
\begin{equation*}
(X^\DOWN,X^\UP)\leq (Y^\DOWN,Y^\UP) \iff 
X^\DOWN\subseteq Y^\DOWN \text{ and } X^\UP\subseteq Y^\UP.
\end{equation*}

The structure of $\mathrm{RS}$ is well studied when
$R$ is an equivalence. J.~Pomykala and {J.A.}~Pomykala 
\cite{PomPom88} proved that $\mathrm{RS}$ is a complete lattice 
forming a Stone algebra. This result was improved in \cite{Comer93} 
by {S.D.}~Comer, who showed that $\mathrm{RS}$ forms a regular double 
Stone algebra. In \cite{GeWa92}, M.~Gehrke and E.~Walker proved that 
$\mathrm{RS}$ is isomorphic to 
$\mathbf{2}^I \times \mathbf{3}^K$, where $I$ is the set of singleton 
$R$-classes, $K$ is the set of non-singleton classes, and 
$\mathbf{2}$ and
$\mathbf{3}$ are the two- and three-element chains, respectively.
In addition, $\mathrm{RS}$ forms a three-valued {\L}ukasiewicz
algebra, as shown by P.~Pagliani \cite{Pagl97}.

Rough sets defined by quasiorders (reflexive and transitive relations) have been investigated by several authors; see 
\cite{Ganter07, JarvRadeRivi24, Kortelainen1994, 
KumBan2015,  Umadevi13, Qiao2012}. If $R$ is a quasiorder,
a Nelson algebra can be defined on $\mathrm{RS}$ \cite{JarRad11}.
Rough sets defined by tolerances (reflexive and symmetric relations)
are studied in \cite{Jarvinen01, JarRad14,SkowStep96, SlowVand00}, 
for example. Generally, for a tolerance, $\mathrm{RS}$ is not 
necessarily a lattice \cite{Jarvinen01}.
In cases where $R$ is a tolerance induced by an irredundant covering 
of $U$, $\mathrm{RS}$ forms a regular pseudocomplemented Kleene
algebra. Although symmetric and transitive relations lead to an $\mathrm{RS}$ structure identical to that of equivalences, 
transitivity alone does not guarantee that $\mathrm{RS}$ is a lattice \cite{Jarv04}. 
Antisymmetric and reflexive relations were also considered  in \cite[Theorem~25]{ManRad20}.

Because $\mathrm{RS}$ lacks a complete lattice structure for certain relations, progress 
in rough set algebraic structures was temporarily delayed.
In \cite{Umadevi2015}, D.~Umadevi presented the
Dedekind--MacNeille completion of $\mathrm{RS}$ for arbitrary binary 
relations, denoted here by $\mathrm{DM(RS)}$. 
Studying completions is important because the ordered set 
$\mathrm{RS}$ is embedded 
within its completion $\mathrm{DM(RS)}$. Consequently, the properties of  the completion also characterise the structure of $\mathrm{RS}$. 
In many cases, $\mathrm{RS}$ is already a complete lattice, which means that it coincides with its completion. 
By focusing on the completion, we simplify the analysis. We no longer need to verify the completeness of 
$\mathrm{RS}$ itself, and this allows for a broader perspective.

In this work, we consider rough set structures induced by reflexive relations.  
Reflexivity can be viewed as an indispensable feature of indiscernibility or similarity, since each object is inherently similar to itself \cite{SlowVand00}.  
In fact, the reflexivity of $R$ is equivalent to a natural
requirement for rough approximations: namely, that 
$X^\DOWN \subseteq X \subseteq X^\UP$ holds for each subset 
$X\subseteq U$. In \cite{SyauJia12}, the approximation theory of 
reflexive neighbourhood systems is studied.

Reflexive relations can be viewed as \emph{directional similarity relations}. A.~Tversky states in \cite{Tversky77} that
similarity should not be treated as a symmetric relation. Statements such as
``$a$ is like $b$'' are directional, with $a$ as 
the subject and $b$ as the referent. This is not equivalent, in general,
to the converse similarity statement ``$b$ is like $a$''. Tversky also provides 
concrete examples, such as
``the portrait resembles the person'' rather than ``the person resembles the
portrait'', and ``the son resembles the father'' rather than ``the father resembles
the son''. It is also clear that similarity relations are
not necessarily transitive.

This study is a continuation of our paper \cite{JarRad25}, where, by describing the
completely join-irreducible elements of $\mathrm{DM(RS)}$ for any reflexive relation $R$, 
we characterised the case when $\mathrm{DM(RS)}$ is a spatial completely distributive
lattice. There, we pointed out that even for a non-transitive reflexive
relation, $\mathrm{DM(RS)}$ can form a Nelson algebra. In the present paper, by
describing the pseudocomplements and dual pseudocomplements in the case of a
completely distributive $\mathrm{DM(RS)}$, we characterise those reflexive relations 
$R$ on $U$ for which $\mathrm{DM(RS)}$ forms a regular pseudocomplemented Kleene algebra or a double Stone lattice.
Note that in this work, we restrict ourselves to the case where 
$\mathrm{DM(RS)}$ is completely distributive and spatial. This requirement is natural because, in the cases where $R$ is 
an equivalence, a quasiorder, or a tolerance induced by an 
irredundant covering, $\mathrm{RS=DM(RS)}$ forms a completely 
distributive and spatial lattice. Therefore, this study can be 
viewed as a generalisation of these aforementioned cases.
One of the main questions that has arisen is whether
the algebraic properties of $\mathrm{RS}$ proven in the case of an equivalence
relation could also occur in the case of a more general relation. 
The results seem to suggest that this is not possible.

Surprisingly, we found a specific class of reflexive relations with 
the property that $\mathrm{DM(RS)}$ forms a regular double Stone lattice, 
which is more general than the class of equivalence relations. A simple example of this is the following: 
For $U = \{1, 2, 3\}$ with $R(1) = U$, $R(2) = \{2\}$, and $R(3) = \{1, 3\}$, 
the lattice $\mathrm{DM(RS)}$ equals $\mathrm{RS}$ and it is a regular double Stone algebra isomorphic to 
$\mathbf{2} \times \mathbf{3}$ (Figure~\ref{fig:two_three}).
\begin{figure}[ht]
\includegraphics[width=45mm]{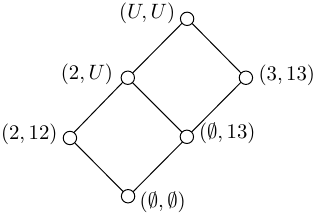}
\caption{
$\mathrm{RS}$ is isomorphic to $\mathbf{2} \times \mathbf{3}$.
\label{fig:two_three}}
\end{figure}
Note that throughout this paper, sets in figures are often 
denoted simply by the sequence of their elements 
(e.g., 123 for $\{1,2,3\}$).

The main motivation for this study stems from the need to establish an algebraic foundation for rough set systems under non-symmetric and non-transitive information environments. While classical Pawlak rough sets are based on equivalence relations, practical soft computing applications, such as directed information networks, asymmetric preference modelling, and directional similarity in feature spaces, inherently violate symmetry and transitivity. By determining the exact structural boundaries under which reflexive approximations yield regular pseudocomplemented Kleene and double Stone algebras, this work addresses a long-standing theoretical gap regarding whether classical algebraic properties can persist in generalised rough set settings.

From an application perspective, characterising these algebraic structures provides crucial mathematical tools for multi-criteria decision analysis (MCDA), incomplete information processing, and formal concept analysis (FCA). Specifically, identifying when $\mathrm{DM(RS)}$ forms a regular Stone or Kleene algebra guarantees the existence of well-behaved logical connectives and negation operators. This allows automated reasoning tools, spatial reasoning systems, and fuzzy-rough rule induction algorithms to operate consistently on asymmetric similarity data without suffering from structural lattice breakdown.

\section{Rough approximations and rough sets} \label{sec:rough_recall}

In this section, we recall from~\cite{Jarvinen2007, JarRad25} the basic results related to the rough approximation operators,
the ordered set of rough sets, and its completion.

Let $U$ be a set. For each $X \subseteq U$, let the lower 
approximation $X^\DOWN$ and the upper approximation $X^\UP$ be defined
as in Section~\ref{sec:intro}. We denote by $\breve{R}$
the \emph{inverse relation} of $R$. The lower and upper
approximations defined by $\breve{R}$ are denoted 
by $X^\Down$ and $X^\Up$, respectively.

Let $\wp(U)$ be the family of all subsets of $U$. 
We define the families of  approximations as follows:
\begin{gather*}
\wp(U)^\DOWN := \{ X^\DOWN \mid X \subseteq U\}, \qquad \wp(U)^\UP := \{ X^\UP \mid X \subseteq U\}, \\
\wp(U)^\Down:= \{ X^\Down \mid X \subseteq U\}, \qquad \wp(U)^\Up := \{ X^\Up \mid X \subseteq U\}. 
\end{gather*}

Let $P$ and $Q$ be ordered sets. A pair
$(f,g)$ of maps $f \colon P \to Q$ and $g \colon Q \to P$
is a \emph{Galois connection} or a \emph{residuated pair} between $P$ and $Q$ if
$f(p) \leq q \iff p \leq g(q)$ for all $p \in P$ and $q \in Q$.
The essential properties and results on Galois connections
can be found in~\cite{Blyth2005,Davey02,Erne_primer}, for example.

The pairs of maps $({^\UP},{^\Down})$ and $({^\Up},{^\DOWN})$ are 
Galois connections on the complete lattice $(\wp(U),\subseteq)$ as noted in~\cite{Jarvinen2007}, for instance. 
This means that if $R$ is an arbitrary binary relation on
$U$, the following facts hold for all $X,Y \subseteq U$ and
$\mathcal{H} \subseteq \wp(U)$: 

\begin{enumerate}[label={\rm (GC\arabic*)}]
    \item $\emptyset^\UP = \emptyset^\Up = \emptyset$ and
    $U^\DOWN = U^\Down = U$.
    \item $X^{\Down\UP} \subseteq X \subseteq X^{\UP\Down}$ and
    $X^{\DOWN\Up} \subseteq X \subseteq X^{\Up\DOWN}$.
    \item $X \subseteq Y$ implies $X^\DOWN \subseteq Y^\DOWN$, 
    $X^\Down \subseteq Y^\Down$, $X^\UP \subseteq Y^\UP$, 
    $X^\Up \subseteq Y^\Up$.
    \item $(\bigcup \mathcal{H})^\UP = 
    \bigcup \{X^\UP \mid X \in \mathcal{H}\}$
    and 
    $(\bigcup \mathcal{H})^\Up = \bigcup\{X^\Up \mid X \in \mathcal{H}\}$.
    \item $(\bigcap \mathcal{H})^\DOWN = 
    \bigcap \{X^\DOWN \mid X \in \mathcal{H}\}$
    and 
    $(\bigcap \mathcal{H})^\Down = \bigcap\{X^\Down \mid X \in \mathcal{H}\}$.
    \item $X^{\UP \Down \UP} = X^\UP$,
    $X^{\Up \DOWN \Up} = X^\Up$,
    $X^{\DOWN \Up \DOWN} = X^\DOWN$,
    $X^{\Down \UP \Down} = X^\Down$.
    \item $(\wp(U)^\UP,\subseteq) \cong (\wp(U)^\Down,\subseteq)$
     and 
    $(\wp(U)^\Up,\subseteq) \cong (\wp(U)^\DOWN,\subseteq)$.
\end{enumerate}
The operations $^\UP$ and $^\DOWN$ are mutually \emph{dual}, and the same holds for $^\Up$ and $^\Down$, that is, 
for any $X \subseteq U$,
\begin{equation} \label{eq:dual}
X^{c \UP} = X^{\DOWN c},
X^{c \DOWN} = X^{\UP c},
X^{c \Up} = X^{\Down c}, 
X^{c \Down} = X^{\Up c},
\end{equation}
where $X^c$ denotes the complement $U \setminus X$ of $X$. 

Because of the duality, $(\wp(U)^\UP,\subseteq) \cong (\wp(U)^\DOWN,\supseteq)$
and $(\wp(U)^\Up,\subseteq) \cong (\wp(U)^\Down,\supseteq)$. Therefore, by combining with (GC7), 
we have the isomorphisms: 
\begin{equation}\label{eq:isomorphisms}
(\wp(U)^\UP,\subseteq) \cong (\wp(U)^\DOWN,\supseteq) \cong
(\wp(U)^\Up,\supseteq) \cong (\wp(U)^\Down,\subseteq).
\end{equation}

If $R$ is a \emph{reflexive} relation
on $U$, that is, $(x,x) \in R$ for all $x \in U$, then
\begin{enumerate}[label={\rm (Ref\arabic*)}]
\item $\emptyset^\DOWN = \emptyset^\Down = \emptyset$ and  
$U^\UP = U^\Up = U$.
\item $X^\DOWN \subseteq X \subseteq X^\UP$ and
$X^\Down \subseteq X \subseteq X^\Up$ for any $X \subseteq U$.
\end{enumerate}
\medskip%
Because $({^\Up}, {^\DOWN})$ is a Galois connection, 
$\wp(U)^\DOWN$ is a complete lattice such that 
\[
\bigwedge_{i \in I} X_i = \bigcap_{i \in I} X_i
\text{\quad and \quad}
\bigvee_{i \in I} X_i = 
\big ( \bigcup_{i \in I} X_i \big)^{\Up\DOWN}
\]
for all $\{X_i\}_{i \in I} \subseteq \wp(U)^\DOWN$.
Similarly, $\wp(U)^\UP$ is a complete lattice such that 
\[
\bigwedge_{i \in I} Y_i = 
\big ( \bigcap_{i \in I} Y_i \big )^{\Down \UP}
\text{\quad and \quad}
\bigvee_{i \in I} Y_i = 
\bigcup_{i \in I} Y_i 
\]
for any $\{Y_i\}_{i \in I} \subseteq \wp(U)^\UP$. We denote by $\mathrm{RS}$ the set of all rough sets, that is,
\[ \mathrm{RS} := \{ (X^\DOWN, X^\UP) \mid X \subseteq U \}. \]
The set $\mathrm{RS}$ is ordered coordinatewise by
\[ (X^\DOWN, X^\UP) \leq (Y^\DOWN, Y^\UP) \overset{\rm def}{\iff} 
X^\DOWN \subseteq Y^\DOWN \ \text{and}  \ X^\UP \subseteq Y^\UP. \]
It is known that $\mathrm{RS}$ is not always a lattice if $R$ is a 
reflexive and symmetric binary relation; see~\cite{Jarvinen01}, for example.

The \emph{Dedekind--MacNeille completion} of an ordered set is the smallest complete lattice containing it; see~\cite{Davey02},
for example. We denote the Dedekind--MacNeille completion of $\mathrm{RS}$ by $\mathrm{DM(RS)}$. 
D.~Umadevi \cite{Umadevi2015} has proved that for any binary relation $R$ on $U$,
\[ \mathrm{DM(RS)} = \{ (A,B) \in \wp(U)^\DOWN \times \wp(U)^\UP \mid A^{\Up \UP} \subseteq B 
\text{ and }  A \cap \mathcal{S} = B \cap \mathcal{S} \} .\]
Here 
\[ \mathcal{S} := 
\{ x \in U \mid R(x) = \{z\} \text{ for some $z \in U$} \}. \]
The elements of $\mathcal{S}$ are called \emph{singletons}.
This means that $x \in \mathcal{S}$ if and only if
$|R(x)| = 1$. Note that if $R$ is reflexive, then $x \in \mathcal{S}$
if and only if $R(x) = \{x\}$.

Next, we present an example that shows how reflexive directional similarity relations arise in many-valued information systems.

\begin{example}
In \cite{Ziarko1993}, W.~Ziarko introduced the relative degree of misclassification of $X$ with 
respect to $Y$ by
\[ c(X,Y) := 
\left \{ 
\begin{array}{ll}
1 - |(X \cap Y)| \, / \, |X|, & \text{if $X \neq \emptyset$}; \\
0, & \text{if $X = \emptyset$},
\end{array}  \right .
\]
where $|Z|$ stands for the cardinality of the set $Z$,
and $X$ and $Y$ are some finite sets.
This is interpreted so that if we were to classify all 
elements of $X$ into set $Y$, then $c(X,Y) \cdot 100$ 
percent of cases would make a classification error.

In standard set theory, $A \subseteq B$ only if every 
element of $A$ is in $B$. In 
\emph{variable precision rough set theory}
(VPRS), this is relaxed using a precision parameter 
$\beta$. The \emph{$\beta$-majority inclusion relation}
is defined as $X \subseteq_\beta Y \iff
c(X,Y) \leq \beta$.

The parameter $\beta$ typically ranges from $0$ to
$0.5$. If $\beta = 0$, the relation 
is the standard set-inclusion relation. For $\beta > 0$,
the relation $\subseteq_\beta$ allows for ``misclassified'' 
elements, making the model more robust against outliers or data errors in large datasets.

In variable precision rough set theory, the set-inclusion
relation is replaced by ${\subseteq}_\beta$ 
when defining the approximations. For an equivalence $E$,
the $\beta$-lower approximation of $X$ is the set of
elements $x$ such that $[x]_E \subseteq_\beta X$,
where $[x]_E$ is the $E$-class of $x$.
Similarly, the $\beta$-negative region of $X$ is the set 
of elements $x$ satisfying $[x]_E \subseteq_\beta X^c$.

We may use the relation ${\subseteq}_\beta$ in
a slightly different way. Let $(U,A)$
be a many-valued information system introduced
by E.~Or{\l}owska and Z.~Pawlak in \cite{OrlPaw84}.
Here, $U$ is the set of objects and  $A$ is the set 
of attributes. Each attribute $a \in A$ is a mapping 
$a \colon U \to \wp(V_a)$, where $V_a$ is
the value set of $a$.

For an attribute $a \in A$, we define a relation $\mathrm{sim}_\beta(a)$ on $U$ 
by setting 
$(x,y) \in \mathrm{sim}_\beta(a) 
\iff a(x) \subseteq_\beta a(y)$.
As noted by Ziarko, $\subseteq_\beta$ is reflexive, but generally neither symmetric nor transitive. Consequently,
$\mathrm{sim}_\beta(a)$ is also reflexive but not necessarily symmetric or transitive. This implies that 
$\mathrm{sim}_\beta(a)$ represents 
\emph{directional similarity}: 
$(x,y) \in \mathrm{sim}_\beta(a)$ indicates that the 
$a$-values of $x$ are included in those of $y$, provided 
that a certain amount of misclassification is allowed. 
Essentially, $x$ is $\beta$-similar to $y$ because $x$ does 
not introduce $a$-values outside of $a(y)$, given the 
$\beta$ tolerance.

For instance, if $a$ is the attribute ``knowledge of
languages'', then $(x,y) \in \mathrm{sim}_\beta(a)$
means that $x$ can be considered similar to $y$ with respect to spoken languages. 
This is because $y$ appears to speak all languages that $x$ speaks, allowing for some uncertainty provided by $\beta$.
\end{example}

Our following lemma shows that for each $A \in \wp(U)^\DOWN$,
there is a pair in $\mathrm{DM(RS)}$ containing $A$ as the first
element. An analogous claim holds for $B \in \wp(U)^\UP$.

\begin{lemma} \label{lem:exists_pair}
Let $R$ be a reflexive relation. 
\begin{enumerate}[label={\rm (\roman*)}]
\item If $A \in \wp(U)^\DOWN$, then $(A,A^{\Up\UP})$ belongs to $\mathrm{DM(RS)}$.
\item If $B \in \wp(U)^\UP$, then $(B^{\Down\DOWN},B)$ belongs to $\mathrm{DM(RS)}$.
\end{enumerate}
\end{lemma}

\begin{proof}
(i) Let $A \in \wp(U)^\DOWN$. Note that $A^{\Up\UP}$ belongs 
to $\wp(U)^\UP$ and $A^{\Up \UP} \subseteq A^{\Up \UP}$ 
holds trivially. Also the inclusion  
$A \cap \mathcal{S} \subseteq A^{\Up \UP} \cap \mathcal{S}$ 
is immediate since $A \subseteq A^{\Up\UP}$.
For the reverse, suppose $x \in A^{\Up \UP} \cap \mathcal{S}$. 
Because $x \in \mathcal{S}$, $\{x\} = R(x) \cap A^\Up$.
Therefore, there is $y \in A \cap \breve{R}(x)$. Since
$A \in \wp(U)^\DOWN$, $A = C^\DOWN$ for some $C \subseteq U$.
Since $y \in A$, $R(y) \subseteq C$. Furthermore, $y \in \breve{R}(x)$ 
implies $x \in R(y)$ and $x \in C$. Given that $R(x) = \{x\}$, we have
$R(x) \subseteq C$ and $x \in C^\DOWN = A$. This confirms
$x \in A \cap \mathcal{S}$. Therefore,
$A \cap \mathcal{S} = A^{\Up \UP} \cap \mathcal{S}$.

Claim (ii) can be proved analogously. 
\end{proof}

The meets and joins are formed in $\mathrm{DM(RS)}$ as
\begin{equation} \label{Eq:Meet}
\bigwedge \{  ( A_i, B_i ) \mid i \in I \}  =
\Big ( \bigcap_{i \in I} A_i, \big ( \bigcap_{i \in I} B_i \big )^{\Down \UP} \Big )
\end{equation}
and
\begin{equation} \label{Eq:Join}
\bigvee \{  ( A_i, B_i ) \mid i \in I \}  =
\Big ( \big ( \bigcup_{i \in I} A_i \big )^{\Up \DOWN} , 
\bigcup_{i \in I} B_i \Big )
\end{equation}
for all $\{(A_i,B_i) \mid i \in I\} \subseteq \mathrm{DM(RS)}$.

A \emph{De~Morgan algebra} $(L,\vee,\wedge,{\sim},0,1)$ is
an algebra such that $(L,\vee,\wedge,0,1)$  
is a bounded distributive lattice and the negation $\sim$ satisfies the \emph{double negation law}
\[ {\sim} {\sim} x = x,\] 
and the two \emph{De Morgan laws} 
\[ {\sim} (x \vee y) = {\sim} x \wedge {\sim} y  
\quad \text{and} \quad 
{\sim} (x \wedge y) = {\sim} x \vee {\sim} y .  \]
A \emph{Kleene algebra} is a De~Morgan algebra satisfying  the inequality
\begin{equation} \label{Eq:Kleene}\tag{K}
x \wedge {\sim} x \leq y \vee {\sim} y.
\end{equation}
Our next proposition is clear by \cite[Lemma~2.1]{JarRad23}.

\begin{proposition} \label{Prop:Kleene}
If $R$ is a reflexive relation such that $\mathrm{DM(RS)}$ is
a distributive lattice, then
\[ (\mathrm{DM(RS)},\vee,\wedge,{\sim},(\emptyset,\emptyset), (U,U) ) \]
is a Kleene algebra in which ${\sim}(A,B) = (B^c, A^c)$ for all
$(A,B) \in  \mathrm{DM(RS)}$.
\end{proposition}

A nonzero element $j$ of a complete lattice $L$ is called 
\emph{completely join-irreducible} if $j = \bigvee S$ implies $j \in S$ 
for every subset $S$ of $L$.
Note that the least element $0 \in L$ is not completely join-irreducible.
The completely join-irreducible elements of $\wp(U)^\UP$ are the
sets $\breve{R}(x)$, $x \in U$, such that for any $A \subseteq U$, 
$\breve{R}(x) = \bigcup \{ \breve{R}(a) \mid a \in A\}$ implies 
$\breve{R}(x) = \breve{R}(b)$ for some $b \in A$. The completely
join-irreducible elements of $\wp(U)^\Up$ are similar sets obtained
by replacing $\breve{R}$ by $R$.

We proved in~\cite{JarRad25} that if $R$ is a reflexive relation on $U$, 
then the set of completely join-irreducible
elements of $\mathrm{DM(RS)}$ is
\begin{gather} \label{eq:Join1}
\{ (\{x\}^{\Up \DOWN},\{x\}^{\Up \UP}) \mid 
\text{$\{x\}^\Up$ is completely join-irreducible in $\wp(U)^\Up$} \} \\
\label{eq:Join2}
\cup \  \{ (\{x\}^\DOWN,\{x\}^\UP) \mid 
\text{$\{x\}^\UP$ is completely join-irreducible in $\wp(U)^\UP$ 
and $x \notin \mathcal{S}$} \}. 
\end{gather}
Moreover, the set of the atoms of $\mathrm{DM(RS)}$ is 
\begin{equation} \label{eq:atoms}
\{ (\{x\}^\DOWN,\{x\}^\UP ) \mid 
\text{$\{x\}^\UP$ is an atom of $\wp(U)^\UP$} \}.
\end{equation}

\section{Pseudocomplements} \label{sec:pseudocomplements}

Let $L$ be a lattice with $0$. The \emph{pseudocomplement} of an element $a \in L$ is the greatest 
element $a^* \in L$ such that $a \wedge a^* = 0$. A lattice $L$ is called \emph{pseudocomplemented} 
if every element $a \in L$ has a pseudocomplement. For a pseudocomplemented lattice $L$,
the algebra $(L,\vee,\wedge,^*,0,1)$ is called a \emph{$p$-algebra}.

Let us start with the following proposition.

\begin{proposition} \label{prop:pseudomplement}
Let $R$ be a reflexive relation.
If $\wp(U)^\UP$ is pseudocomplemented, then $\mathrm{DM(RS)}$
is also pseudocomplemented. For $(A,B) \in \mathrm{DM(RS)}$, 
\[ (A,B)^* = (B^{* \Down \DOWN}, B^*),\]
where $B^*$ is the pseudocomplement of $B$ in $\wp(U)^\UP$.
\end{proposition}

\begin{proof} By Lemma~\ref{lem:exists_pair}, $(B^{* \Down \DOWN}, B^*)$
belongs to $\mathrm{DM(RS)}$. 
We first prove that 
\begin{equation} \label{eq:meet_for_pseudo}
 (A,B) \wedge (B^{* \Down \DOWN}, B^*) = 
  (A \wedge B^{* \Down \DOWN}, B \wedge B^*)
  = (\emptyset,\emptyset),
\end{equation}
where the meets are taken component-wise in 
$\wp(U)^\DOWN$ and $\wp(U)^\UP$, respectively.
Obviously, $\emptyset = B \wedge B^* = 
(B \cap B^*)^{\Down \UP}$ in $\wp(U)^\UP$.
Because $A^{\Up \UP} \subseteq B$, we obtain
$A^\Up \subseteq A^{\Up (\UP \Down)} \subseteq B^\Down$ and 
$A \subseteq A^{\Up\DOWN} \subseteq B^{\Down \DOWN}$.
Thus, in $\wp(U)^\DOWN$,
\[ A \wedge B^{* \Down \DOWN} =
A \cap B^{* \Down \DOWN} \subseteq
B^{\Down \DOWN} \cap B^{* \Down \DOWN} =
(B \cap B^*)^{\Down \DOWN}
\subseteq (B \cap B^*)^{\Down \UP}
= \emptyset.
\]
This proves \eqref{eq:meet_for_pseudo}.

Secondly,  let $(X,Y) \in \mathrm{DM(RS)}$ be such that 
$(A,B) \wedge (X,Y) = \emptyset$. The condition
$B \wedge Y = \emptyset$ in $\wp(U)^\UP$ implies
that $Y \subseteq B^*$. Since $(X,Y) \in \mathrm{DM(RS)}$, we have
$X^{\Up \UP} \subseteq Y \subseteq B^*$ and
$X^\Up \subseteq X^{\Up (\UP \Down)} \subseteq Y^\Down \subseteq 
B^{* \Down}$. It follows
$X \subseteq X^{\Up\DOWN} \subseteq B^{*\Down \DOWN}$.
We have proved $(X,Y) \leq (B^{* \Down \DOWN}, B^*)$,
completing the proof.
\end{proof}

A \emph{pseudocomplemented De Morgan algebra} is an algebra 
$(L,\vee,\wedge,{\sim},^*,0,1)$ such that $(L,\vee,\wedge,{\sim},0,1)$ is a De Morgan
algebra and $(L,\vee,\wedge,^*,0,1)$ is a $p$-algebra. In fact, such an
algebra forms a double $p$-algebra, where the pseudocomplement operations
determine each other by
\begin{equation}\label{eq:KleeneRule}
{\sim} x^*=({\sim} x)^+ \quad \text{and} \quad {\sim} x^+=({\sim} x)^*.
\end{equation}
A \emph{pseudocomplemented Kleene algebra} is defined analogously.

A complete lattice $L$ is \emph{completely distributive} if for any doubly indexed subset $\{a_{i,\,j}\}_{i \in I, \, j \in J}$ of $L$, we have:
\[
\bigwedge_{i \in I} \Big ( \bigvee_{j \in J} a_{i,\,j} \Big ) = 
\bigvee_{ f \colon I \to J} \Big ( \bigwedge_{i \in I} a_{i, \, f(i) } \Big ), \]
that is, any meet of joins may be converted into the join of all possible elements obtained by taking the meet over $i \in I$ of
elements $a_{i,\,k}$\/, where $k$ depends on $i$. 
Any completely distributive lattice $L$ is pseudocomplemented and
\begin{equation} \label{eq:pseudocomplemented}
a^* = \bigvee \{ x \in L \mid a \wedge x = 0\}.    
\end{equation}

We can now write a corollary of Propositions~\ref{Prop:Kleene} 
and~\ref{prop:pseudomplement},
and~\cite[Proposition~{4.1}]{JarRad25}

\begin{corollary} \label{cor:Kleene_pseudocompl}
Let $R$ be a reflexive relation on $U$ such that 
$\wp(U)^\UP$ is completely distributive. 
Then, $\mathrm{DM(RS)}$ forms a 
pseudocomplemented Kleene algebra such that 
\[{\sim}(A,B) = (B^c,A^c) \quad \text{and} \quad
(A,B)^* = (B^{* \Down \DOWN}, B^*)
\]
for all $(A,B) \in \mathrm{DM(RS)}$, where
$B^*$ is the pseudocomplement defined in $\wp(U)^\UP$. 
The underlying
lattice $\mathrm{DM(RS)}$ is completely distributive.
\end{corollary}

In what follows, we present a description of $B^*$ for 
$B \in \wp(U)^\UP$ using the notion of the core. We recall its 
definition from~\cite{JarRad25}. Let $R$ be a binary relation on 
$U$ and $x \in U$. The \emph{core of $R(x)$} is defined by 
\begin{equation} \label{eq:core} 
\mathfrak{core}R(x) := 
\{w \in R(x) \mid w \in R(y) \implies R(x) \subseteq R(y) \}. \end{equation} 
Similarly, we define the \emph{core of $\breve{R}(x)$} by replacing 
$R$ with $\breve{R}$ in \eqref{eq:core}. The properties of the 
core are listed in \cite{JarRad25}.
We define the following set for any
$B \in \wp(U)^\UP$:
\[ \mathcal{K}(B) := \bigcup 
\{ \mathfrak{core}\breve{R}(x) \mid \breve{R}(x) 
\subseteq B \} . \]
We can now compute: 
\begin{align*}
y \in \mathcal{K}(B)^\Up & \iff
\breve{R}(y) \cap \mathcal{K}(B) \neq \emptyset  \\
& \iff \breve{R}(y) \cap 
\textstyle{\bigcup} \{ \mathfrak{core}\breve{R}(x) \mid 
\breve{R}(x) \subseteq B \} \neq \emptyset  \\
& \iff \textstyle{\bigcup} \{  \breve{R}(y) \cap  
\mathfrak{core}\breve{R}(x) \mid \breve{R}(x) 
\subseteq B \} \neq \emptyset  \\
& \iff \text{there is } \breve{R}(x) 
\subseteq B \text{ such that }  \breve{R}(y) \cap  
\mathfrak{core}\breve{R}(x) \neq \emptyset.
\end{align*}
Trivially, $\breve{R}(y) \cap \mathfrak{core}\breve{R}(x) \neq \emptyset$ 
means that there exists an element 
$z \in \mathfrak{core}\breve{R}(x) $ such 
that $z \in \breve{R}(y)$. 
By the definition of the core, this implies
that $\breve{R}(x) \subseteq \breve{R}(y)$ and 
$\mathfrak{core}\breve{R}(x) \neq \emptyset$. 
On the other hand, $\mathfrak{core}\breve{R}(x) \neq \emptyset$
implies that there is $z \in \mathfrak{core}\breve{R}(x)$.
Because $\mathfrak{core}\breve{R}(x) \subseteq \breve{R}(x)
\subseteq \breve{R}(y)$, we have $z \in \breve{R}(y)$. Thus,
\[ 
y \in \mathcal{K}(B)^\Up  \iff
\text{there is } \breve{R}(x) 
\subseteq B \text{ such that }  \breve{R}(x) \subseteq \breve{R}(y) 
\text{ and } \mathfrak{core}\breve{R}(x) \neq \emptyset.
\]

The set of completely join-irreducible elements of $L$ is denoted by 
$\mathcal{J}(L)$, or simply by $\mathcal{J}$ if there is no danger of 
confusion.
A complete lattice $L$ is \emph{spatial} if for each $a \in L$,
\[ a = \bigvee \{ j \in \mathcal{J} \mid j \leq a \}. \]

An element $p$ of a complete lattice $L$ is said to be 
\emph{completely join-prime} if for every $X \subseteq L$, $p \leq \bigvee X$ implies $p \leq x$ for some $x \in X$. 
We denote by $\mathcal{J}_p(L)$ the set of all completely
join-prime elements of $L$.

In a complete lattice $L$, each completely join-prime element is completely join-irreducible. 
The converse does not always hold, but if $L$ is completely distributive, then the sets 
of completely join-prime and completely join-irreducible elements coincide~\cite{Balachandran55}.

Let $R$ be such that $\wp(U)^\UP$ is completely distributive 
and spatial. It is noted in \cite[Lemma~4.10]{JarRad25} 
that $\mathfrak{core}\breve{R}(x) \neq \emptyset$
is equivalent to $\{x\}^\UP$ being completely join-prime.
Because in a completely distributive lattice, the sets of
completely join-prime elements and completely join-irreducible
elements coincide, we have that $\mathfrak{core}\breve{R}(x) \neq \emptyset$ if and only if $\{x\}^\UP \in \mathcal{J}(\wp (U)^{\UP})$,
since $\breve{R}(x)$ equals $\{x\}^\UP$. We can now write:
\[ y \in \mathcal{K}(B)^\Up  \iff 
\text{there is } \{x\}^\UP \subseteq B \text{ such that }  
\{x\}^\UP \in \mathcal{J}(\wp (U)^{\UP}) \text{ and }
\{x\}^\UP \subseteq \{y\}^\UP.
\]
Negating this condition and replacing $x$ with $j$ for clarity yields:
\[ y \in \mathcal{K}(B)^{\Up c} \iff \{j\}^\UP \nsubseteq \{y\}^\UP 
\text{ for all } \{j\}^\UP \in \mathcal{J}(\wp (U)^{\UP})
\text{ such that }  \{j\}^\UP \subseteq B.
\]

On the other hand, by \eqref{eq:pseudocomplemented},
\[
B^* = \bigcup \{ Y^\UP \mid Y^\UP \wedge B = \emptyset\}. 
\]
We can also express the pseudocomplement in the form
\begin{equation} \label{eq:pseudo_in_up}
B^* = \bigcup \{ \{y\}^\UP \mid \{y\}^\UP \wedge B =\emptyset \}. 
\end{equation}
This is because
$\{\{y\}^\UP \mid \{y\}^\UP \wedge B =\emptyset\} \subseteq
 \{ Y^\UP \mid Y^\UP \wedge B =\emptyset\}$. Hence,
\[ \bigcup \{\{y\}^\UP \mid \{y\}^\UP \wedge B =\emptyset\} 
\subseteq \bigcup \{ Y^\UP \mid Y^\UP \wedge B = \emptyset\}
\subseteq B^*.\]
Since $\{y\}^\UP \subseteq B^*$ implies $\{y\}^\UP \wedge B = \emptyset$,
we have 
\[
B^* = \bigcup \{ \{y\}^\UP \in \mathcal{J}(\wp(U)^\UP \mid
\{y\}^\UP \subseteq B^* \} \subseteq  
\bigcup \{\{y\}^\UP \mid \{\{y\}^\UP \wedge B =\emptyset\}.
\]
Therefore, \eqref{eq:pseudo_in_up} holds.

We can now state the following proposition.

\begin{proposition} \label{prop:upper_pseudo}
Let $R$ be a reflexive relation on $U$ such that 
$\wp(U)^\UP$ is completely distributive 
and spatial. Then, 
\begin{equation} \label{eq:pseudocomplement_formula}    
B^* = \mathcal{K}(B)^{\Up c \UP}. 
\end{equation}
\end{proposition}

\begin{proof} By \eqref{eq:pseudo_in_up},
\[ B^* = \bigcup \{\{y\}^\UP \mid \{\{y\}^\UP \wedge B = \emptyset\}.\]
For any $\{y\}^\UP$, $B \wedge \{y\}^\UP \in \wp (U)^\UP$.
Because $\wp(U)^\UP$ is spatial, $B \wedge\{y\}^\UP =\emptyset$ is
equivalent to that there is no  
$\{j\}^\UP \in \mathcal{J}(\wp(U)^\UP)$ with  $\{j\}^\UP \subseteq B$ and 
$\{j\}^\UP \subseteq \{y\}^\UP$. This is because $\{j\}^\UP \subseteq B$
and $\{j\}^\UP \subseteq \{y\}^\UP$ would imply 
$\emptyset \neq \{j\}^\UP \subseteq B \wedge \{y\}^\UP$. On the other
hand, if $B \wedge \{y\}^\UP \neq \emptyset$, there is 
$\{j\}^\UP \in \mathcal{J}(\wp(U)^\UP)$ such that
$\{j\}^\UP \subseteq B \wedge \{y\}^\UP$, and $B \wedge \{y\}^\UP$ is included both in $B$ and $\{y\}^\UP$. Therefore, 

\begin{align*}
B^* &= \bigcup \{ \{y\}^\UP \mid \{j\}^\UP \nsubseteq \{y\}^\UP 
\text{ for all } 
\{j\}^\UP \in \mathcal{J}(\wp(U)^\UP) \text{ such that }
\{j\}^\UP \subseteq B\} \\
&= \{ y \mid \{j\}^\UP \nsubseteq \{y\}^\UP \text{ for all } 
\{j\}^\UP \in \mathcal{J}(\wp(U)^\UP) \text{ such that } 
\{j\}^\UP \subseteq B \}^\UP \\
&= \mathcal{K}(B)^{\Up c \UP}. \qedhere
\end{align*}
\end{proof}

Next, we present an example showing that the converse of Proposition~\ref{prop:upper_pseudo} does not hold. 
We define a reflexive relation such that $\wp(U)^\UP$ is a pseudocomplemented lattice and satisfies  \eqref{eq:pseudocomplement_formula}, 
but $\wp(U)^\UP$ is not distributive.

\begin{example}
Let $U = \{1, 2, 3, 4\}$ and let $R$ be such that 
\[ 
R(1) = \{1, 4\}, \quad R(2) = R(3) = \{1, 2, 3\} \quad R(4) = \{4\}.
\]
Then,
\[
\breve{R}(1) = \{1, 2, 3\}, \quad \breve{R}(2) = \breve{R}(3) = \{2, 3\}, \quad \breve{R}(4) = \{1, 4\}.
\]
The elements of $\wp(U)^\UP$ are $\emptyset, \{2, 3\}, \{1, 2, 3\}, \{1, 4\}, U$. Obviously, $\wp(U)^\UP$ is isomorphic to the lattice $\mathbf{N_5}$ 
which is pseudocomplemented, but not distributive. 
We have 
\[ \{1, 2, 3\}^* = \{2, 3\}^* = \{1, 4\}, \quad  \{1, 4\}^* = \{1, 2, 3\}, \quad \emptyset^* = U, \quad U^* = \emptyset. \] 
The $\breve{R}$-cores are:
\[
\mathfrak{core}\breve{R}(1) = \emptyset, \quad \mathfrak{core}\breve{R}(2) = \mathfrak{core}\breve{R}(3) = \{2, 3\}, \quad 
\mathfrak{core}\breve{R}(4) = \{4\}.
\]
Since $\mathcal{K}(B) = \bigcup 
\{ \mathfrak{core}\breve{R}(x) \mid \breve{R}(x) \subseteq B\}$ 
for each $B \in \wp(U)^\UP$, we have
\[
\mathcal{K}(\emptyset) = \emptyset, \quad 
\mathcal{K}(\{2, 3\}) = \mathcal{K}(\{1, 2, 3\}) = \{2, 3\}, \quad 
\mathcal{K}(\{1, 4\}) = \{4\}, \quad
\mathcal{K}(U) = \{2, 3, 4\}.
\]
We obtain:
\begin{align*}
\mathcal{K}(\emptyset)^\Up &= \emptyset^\Up = \emptyset, \\
\mathcal{K}(\{2, 3\})^\Up &= \mathcal{K}(\{1, 2, 3\})^\Up = 
\{2, 3\}^\Up = \{1, 2, 3\}, \\
\mathcal{K}(\{1, 4\})^\Up &= \{4\}^\Up = \{4\}, \\ 
\mathcal{K}(U)^\Up &= \{2, 3, 4\}^\Up = U.
\end{align*}
Finally,
\begin{align*}
\mathcal{K}(\emptyset)^{\Up c \UP} &= \emptyset^{c \UP} = U^\UP = U, \\ 
\mathcal{K}(\{2, 3\})^{\Up c \UP} &= \mathcal{K}(\{1, 2, 3\})^{\Up c \UP} = \{1, 2, 3\}^{c \UP} = \{4\}^\UP = \{1,4\}, \\ 
\mathcal{K}(\{1, 4\})^{\Up c \UP} &= \{4\}^{c \UP} = \{1,2,3\}^\UP = \{1,2,3\}, \\
\mathcal{K}(U)^{\Up c \UP} &=  U^{c \UP} = \emptyset^\UP = \emptyset.
\end{align*}
We have now shown that $B^* = \mathcal{K}(B)^{\Up c \UP}$ holds in the pseudocomplemented lattice $\wp(U)^\UP$.
Because  $\wp(U)^\UP$ is not distributive, the converse of Proposition~\ref{prop:upper_pseudo} does not hold. 
\end{example}

\begin{remark}
Note that we can write $\mathcal{K}(B)^{\Up c \UP}$ in the
form $\mathcal{K}(B)^{c\Down \UP}$. Since $\wp(U)^\UP$ is
the interior system corresponding to the interior operator
$X \mapsto X^{\Down \UP}$, we can write the pseudocomplement in
the form 
\begin{equation} \label{eq:pseudocomplement_upper}
B^* = \bigcup \{ X \in \wp(U)^\UP \mid X \subseteq \mathcal{K}(B)^c\}.
\end{equation}
\end{remark}

\begin{example} \label{exa:pseudocomplements}
Let $R$ be a reflexive relation on $U = \{1,2,3,4\}$ such that
\begin{gather*}
R(1) = \{1,2\}, \ R(2) = \{1,2,3\}, \ R(3) = \{3\}, \ 
R(4) = \{1,3,4\}.
\end{gather*}
\begin{figure}[htbp] 
\centering 
\includegraphics[width=0.8\textwidth]{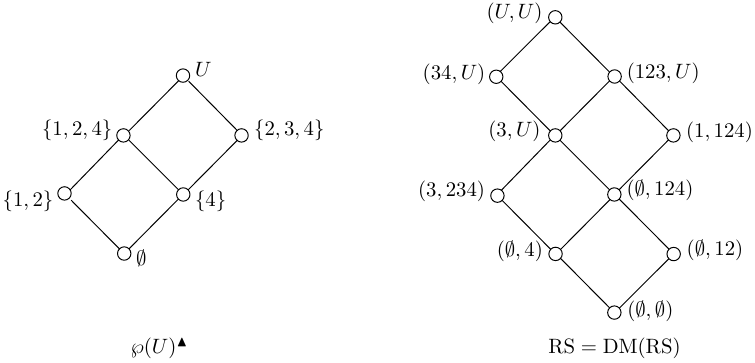} 
\caption{The Hasse diagrams of $\wp(U)^\UP$ and 
$\mathrm{RS} = \mathrm{DM(RS)}$.} 
\label{fig:pseudocomplements} 
\end{figure}

The inverse relation $\breve{R}$ has the neighbourhoods:
\begin{gather*}
\breve{R}(1) = \{1, 2, 4\}, \
\breve{R}(2) = \{1, 2\}, \ 
\breve{R}(3) = \{2, 3, 4\}, \
\breve{R}(4) = \{4\}.
\end{gather*}
The $\breve{R}$-cores are:
\[
\mathfrak{core}\breve{R}(1) = \emptyset, \
\mathfrak{core}\breve{R}(2) = \{1\}, \
\mathfrak{core}\breve{R}(3) = \{3\}, \
\mathfrak{core}\breve{R}(4) = \{4\}.
\]

The Hasse diagram of $\wp(U)^\UP$ is given in
Figure~\ref{fig:pseudocomplements}~(left). Obviously,
$\wp(U)^\UP$ is distributive. Because it is finite, it
is spatial and completely distributive. 
Now, we can compute the pseudocomplement of $B = \{2,3,4\}$. 
Since $\breve{R}(3)$ and $\breve{R}(4)$ are included in 
$B$, it follows that
\[ \mathcal{K}(B) = \mathfrak{core}\breve{R}(3) \cup
\mathfrak{core}\breve{R}(4) =  \{3\} \cup \{4\} = \{3,4\}.\]
Therefore, $\mathcal{K}(B)^c = \{1,2\}$. This set
is itself the greatest element of $\wp(U)^\UP$ included in $\{1,2\}$,
so $B^* = \{1,2\}$ by \eqref{eq:pseudocomplement_upper}.

Note that $\mathcal{K}(B)^c$ is truly needed instead of $B^c$.
This is because $B^c = \{1\}$ and the greatest element of
$\wp(U)^\UP$ included in it is $\emptyset$, which is not the pseudocomplement of $B$.

The ordered set $\mathrm{RS}$ is
itself a lattice, so $\mathrm{DM(RS)} = \mathrm{RS}$. Its
Hasse diagram can be found in 
Figure~\ref{fig:pseudocomplements}~(right).
Let $(A,B) = (\{3\}, \{2,3,4\})$. By 
Proposition~\ref{prop:pseudomplement}, 
$(A,B)^* = (B^{* \Down \DOWN}, B^*)$. We just computed that
$B^* = \{1,2\}$. Since 
\[ B^{*\Down \DOWN} = \{1,2\}^{\Down \DOWN} = 
\{2\}^\DOWN = \emptyset,\]
we have that $ (\{3\}, \{2,3,4\})^* = (\emptyset, \{1,2\})$.
\end{example}

In the following example, we consider pseudocomplements of rough sets and approximations 
defined by quasiorders (reflexive and transitive binary relations).

\begin{example} \label{ex:quasiorder_pseudo}
Let $R$ be a quasiorder on $U$. Then 
$x \in \mathfrak{core}R(x)$ for any $x \in U$. 
This follows from the fact that if $x \in R(y)$, then  
$R(x) \subseteq R(y)$ by transitivity. Analogously, 
$x \in \mathfrak{core}\breve{R}(x)$.

We previously considered rough sets defined by quasiorders 
in \cite{JarRadVer09}.
If $R$ is a quasiorder on $U$, then the operators
$^\UP$ and $^\Up$ are closure operators. This implies that
$X^{\UP \UP} = X^\UP$ for all $X \subseteq U$.
In addition, the following identities hold:
\[ X^{\UP\Down} = X^\UP, \quad
X^{\Up\DOWN} = X^\Up, \quad
X^{\DOWN\Up} = X^\DOWN, \quad
X^{\Down\UP} = X^\Down.
\]

Let $B \in \wp(U)^\UP$. Then $B = X^\UP$ for some
$X \subseteq U$ and $B^\UP = B$. 
This also yields that 
$x \in B$ if and only if $\breve{R}(x) \subseteq B$.
For all $x \in B$, we have 
$x \in \mathfrak{core}\breve{R}(x) \subseteq
\breve{R}(x) \subseteq B$. We obtain
$B \subseteq 
\bigcup \{ \mathfrak{core}\breve{R}(x) \mid x \in B\} \subseteq B$ and 
\[
B = \bigcup \{\mathfrak{core}\breve{R}(x) \mid x \in B\} 
= \bigcup \{\mathfrak{core}\breve{R}(x) \mid 
\breve{R}(x) \subseteq B\} 
= \mathcal{K}(B).
\]
Using Proposition~\ref{prop:upper_pseudo}, we have
\[B^* =  \mathcal{K}(B)^{\Up c \UP} 
= B^{\Up c \UP} =  B^{\Up \DOWN c} = B^{\Up c}.
\]
Thus, for $(A,B) \in \mathrm{RS}$,
\[
(A,B)^* = (B^{* \Down \DOWN}, B^*) =
(B^{\Up c \Down \DOWN}, B^{\Up c}) =
(B^{\Up \Up \UP c}, B^{\Up c}) =
(B^{ \Up \UP c}, B^{\Up c});
\]
see also~\cite[p.~409]{JarvRadeRivi24}.
\end{example}

\begin{remark}
The lattice $(\wp(U)^\UP, \subseteq)$ is isomorphic to 
$(\wp(U)^\DOWN, \supseteq)$ by the map $\varphi \colon B \mapsto B^c$. Similarly, the map $\phi \colon A \mapsto A^c$ is an isomorphism between 
$(\wp(U)^\DOWN, \supseteq)$ and $(\wp(U)^\UP, \subseteq)$. 
This implies that $\wp(U)^\UP$ is pseudocomplemented if and only if 
$\wp(U)^\DOWN$ is dually pseudocomplemented with respect to set-inclusion order.

Suppose that $\wp(U)^\UP$ is pseudocomplemented.
For $B \in \wp(U)^\UP$, $\varphi(B^*) = 
\varphi(B)^+$ and for $A \in \wp(U)^\DOWN$, 
$\psi(A^+) = \psi(A)^*$. 
Let $A \in \wp(U)^\DOWN$. Then $A = \varphi(B)$ and $B = \psi(A)$
for some $B \in \wp(U)^\UP$. We have
\[  A^+ = \varphi(B)^+ = \varphi(B^*) = B^{* c} =
     \psi(A)^{* c} = A^{c * c}.\]
Similarly, $B^* = B^{c+c}$ for all $B \in \wp(U)^\UP$.

In $\mathrm{DM(RS)}$, $(A,B)^+ = {\sim}({\sim}(A,B)^*)$. Now,
$({\sim}(A,B))^* = (B^c,A^c)^* = (A^{c * \Down \DOWN}, A^{c *})$.
Thus, 
\[ (A,B)^+ =
{\sim}(A^{c * \Down \DOWN}, A^{c *}) =
(A^{c * c}, A^{c * \Down \DOWN c} )
= (A^{c * c}, A^{c * c \Up \UP} ) =
(A^+, A^{+\Up \UP}),
\]
where $A^+$ is the pseudocomplement in $\wp(U)^\DOWN$.
\end{remark}

\section{Regularity} \label{sec:regular}

An algebra $(L,\vee,\wedge,^*,^{+},0,1)$ is called a \emph{double $p$-algebra} 
if $(L,\vee,\wedge,^*,0,1)$ is a $p$-algebra and $(L,\vee,\wedge,^{+},0,1)$ is a 
\emph{dual $p$-algebra} (that is, $z\geq
x^{+} \iff x\vee z=1$ for all $x,z \in L$). In this case, 
$L$ is called a \emph{double pseudocomplemented} lattice. 

We say that a double $p$-algebra is \emph{regular} if it satisfies the 
condition
\begin{equation} \label{eq:condM} \tag{M}
x^* = y^* \text{ and } x^+ = y^+ \text{ imply } x = y. 
\end{equation}
Here ``regularity'' refers to ``congruence-regularity''. 
An algebra is \emph{congruence-regular} if every congruence is determined 
by any class of it: two congruences are necessarily equal when they
have a  class in common. J.~Varlet has proved in \cite{Var72} that
double  pseudocomplemented lattices satisfying \eqref{eq:condM} are
exactly the congruence-regular ones.

Sankappanavar has proved in \cite{Sankappanavar86} that any 
pseudocomplemented De~Morgan algebra satisfying (M) is congruence-regular. 
Therefore, we may call pseudocomplemented
De~Morgan and Kleene algebras regular when they satisfy (M).  

In this section, we give conditions under which 
$\mathrm{DM(RS)}$ forms a regular pseudocomplemented Kleene algebra.
We say that the set $\mathcal{J}$ of completely join-irreducible 
elements of a complete lattice $L$ \emph{has at most two levels} 
if for any  $j,k \in \mathcal{J}$, $j<k$ implies that $j$
is an atom. It is easy to check that in the case of a spatial 
lattice $L$, $\mathcal{J}$ has at most two levels if and only if 
$\mathcal{J}$ contains no chains of three or more elements. 

In  \cite[Proposition 4.4]{JarRad18}, we presented the following
equivalence. 

\begin{proposition} \label{prop:at_most_two}
Let $(L,\vee,\wedge,{\sim},{^*},0,1)$ be a
pseudocomplemented De~Morgan algebra defined on an algebraic lattice. The
following are equivalent:
\begin{enumerate}[label={\rm (\roman*)}]
\item  $(L,\vee,\wedge,{\sim},{^*},0,1)$ is regular;
\item $\mathcal{J}$ has at most two levels.
\end{enumerate}    
\end{proposition}
\noindent%
Recall from \cite[Theorem~{10.29}]{Davey02} that for a De~Morgan algebra $(L,\vee,\wedge,{\sim},0,1)$,
being algebraic is equivalent to $L$ being a spatial and completely distributive lattice.

Let $(L,\vee,\wedge,{\sim},0,1)$ be a completely 
distributive De Morgan algebra. We define, for any 
$j \in \mathcal{J}$, the element
\begin{equation}\label{Eq:Gee}
 g(j) := \bigwedge \{x \in L \mid x\nleq {\sim} j \}.
\end{equation}
This $g(j) \in \mathcal{J}$ is the least element that is 
\emph{not} below ${\sim}j$. 

Let us recall from \cite{JarRad11, JarRad18}
the essential properties of the function 
$g \colon \mathcal{J} \to \mathcal{J}$.
It satisfies the following conditions:
\begin{enumerate}[label = ({J}\arabic*)]
 \item if $x \leq y$, then $g(x) \geq g(y)$;
 \item $g(g(x))= x$.
\end{enumerate}
If $(L,\vee,\wedge,{\sim},0,1)$ is a completely distributive 
Kleene algebra, then $j$ and $g(j)$ are comparable for any 
$j \in \mathcal{J}$, that is,
\begin{enumerate}[label = ({J}\arabic*)]
\addtocounter{enumi}{2}
 \item $g(j)\leq j \text{ or } j \leq g(j)$.
\end{enumerate}
We may define three disjoint sets:
\begin{align*}
    \mathcal{J}^- &= \{j \in \mathcal{J} \mid j < g(j) \}; \\
    \mathcal{J}^\circ &= \{j \in \mathcal{J} \mid j = g(j) \}; \\
    \mathcal{J}^+ &= \{j \in \mathcal{J} \mid j > g(j) \}.
\end{align*}
Note that 
$\mathcal{J}^{-} = \{ j \in \mathcal{J} \mid  j \leq {\sim} j\}$. Furthermore, the involution $g$ provides a bijection between 
$\mathcal{J}^-$ and $\mathcal{J}^+$; specifically,
$j\in \mathcal{J}^{-} \iff  g(j)\in \mathcal{J}^{+}$.
In \cite[Lemma 4.5]{JarRad18} we showed
that each element of $\mathcal{J}^{\circ}$ is
incomparable with all other elements of $\mathcal{J}$.

Let $R$ be a reflexive relation. As stated in 
Corollary~\ref{cor:Kleene_pseudocompl}, if 
$\mathrm{DM(RS)}$ is completely distributive, then
\[ 
(\mathrm{DM(RS)},\vee,\wedge,{\sim},{^*}, 
(\emptyset,\emptyset), (U,U) ) 
\]
is a pseudocomplemented Kleene algebra in which ${\sim}(A,B) = (B^c, A^c)$ 
and $(A,B)^* = (B^{* \Down \DOWN}, B^*)$, where $B^*$ is the 
pseudocomplement of $B$ in $\wp(U)^\UP$. 

The following result appeared in \cite[Proposition 5.3]{JarRad25}.

\begin{proposition} \label{prop:partition_of_J}
Let $R$ be a reflexive relation on $U$ such that 
$\mathrm{DM(RS)}$ is completely distributive. The
following assertions hold:
\begin{enumerate}[label={\rm (\roman*)}, itemsep=4pt]
\item %i
$\mathcal{J}^{-} = 
\{(\emptyset ,\{x\}^\UP) \mid \{x\}^\UP \in
\mathcal{J}(\wp (U)^{\UP}) \text{ and } x \notin \mathcal{S} \}$;

\item %ii
If $(\emptyset ,\{x\}^\UP) \in \mathcal{J}^{-}$, then
$g(\emptyset,\{x\}^\UP) =  (\{z\}^{\Up \DOWN}, \{z\}^{\Up \UP})$
for any $z\in \mathfrak{core}\breve{R}(x)$, $z \notin \mathcal{S}$,
and $\{z\}^\Up$ is completely join-irreducible in  
$\wp(U)^\Up$;

\item %iii
$\mathcal{J}^{+} = \{(\{x\}^{\Up \DOWN},\{x\}^{\Up \UP}) \mid 
\{x\}^\Up \in \mathcal{J}(\wp(U)^\Up) \text{ and }
x \notin \mathcal{S} \}$;

\item %iv
$\mathcal{J}^{\circ}$ = $\{(\{x\},\{x\}^\UP)\mid x\in \mathcal{S} \}$.
\end{enumerate}
\end{proposition}

Let $P$ be an ordered set with a least element $0$. $ P$ is called \emph{atomic} if 
every element $b>0$ has an atom below it, and $P$ is \emph{atomistic} if every element of $P$ is 
the join of atoms below it. The set of atoms of $P$ is denoted by
$\mathcal{A}(P)$, or simply by $\mathcal{A}$ if there
is no danger of confusion.
We can now present the following correspondences.

\begin{theorem} \label{thm:Stone_vs_Boolean}
Let $R$ be a reflexive relation on $U$. 
The following are equivalent:
\begin{enumerate}[label={\rm (\roman*)}]
\item  $\mathrm{DM(RS)}$ is a spatial, completely distributive, and regular 
pseudocomplemented Kleene algebra.
\item $\mathrm{DM(RS)}$ is a spatial and completely distributive lattice 
such that $\mathcal{J}$ has at most two levels.
\item $\wp(U)^\UP$ is an atomic Boolean lattice.
\end{enumerate}  
\end{theorem}
    
\begin{proof} 
If $\mathrm{DM(RS)}$ is completely distributive, as in (i) and (ii), 
then it is pseudocomplemented, and it forms a Kleene algebra as described in 
Proposition~\ref{Prop:Kleene}. 
Moreover, (i) and (ii) 
are equivalent by Proposition~\ref{prop:at_most_two}, 
whenever $\mathrm{DM(RS)}$ is spatial and completely distributive.

\medskip\noindent%
$\mathrm{(ii)}\Rightarrow\mathrm{(iii)}$: 
Let $\mathrm{DM(RS)}$ be a spatial and completely distributive lattice and suppose that $\mathcal{J}$ has at most two levels.
We assume for contradiction that $\wp(U)^\UP$ is not a 
Boolean lattice. 
Since $\wp(U)^\UP$ is completely distributive and spatial, this
means that there exists a completely join-irreducible element $\{x\}^\UP$ of 
$\wp(U)^\UP$ that is not an atom. Hence, there exists $\{z\}^\UP \in \wp(U)^\UP$ with
$\{z\}^\UP \subset \{x\}^\UP$. Observe that $z$ can not
be a singleton. Indeed, $z\in\{z\}^\UP \subset
\{x\}^\UP$ and $z\in \mathcal{S}$ would imply $x\in R(z)=\{z\}$. So, we
would get $x = z$ and $\{x\}^\UP = \{z\}^\UP$, a contradiction. Hence, in
view of Proposition~\ref{prop:partition_of_J}, $(\{z\}^\DOWN,\{z\}^{\UP}) =
(\emptyset,\{z\}^\UP)$ belongs to  $\mathcal{J}^{-}$. 

If $x\in \mathcal{S}$, then $(\{x\},\{x\}^\UP)\in \mathcal{J}^{\circ}$ and
\[ (\emptyset,\{z\}^\UP) < (\{x\},\{x\}^\UP) = g(\{x\},\{x\}^\UP)
<  g(\emptyset,\{z\}^\UP).\]
On the other hand, if $x \notin \mathcal{S}$, then
$(\emptyset,\{x\}^\UP) \in \mathcal{J}^{-}$ by Proposition~\ref{prop:partition_of_J}. Now,
\[ (\emptyset,\{z\}^\UP) < (\emptyset,\{x\}^\UP) < g(\emptyset,\{x\}^\UP). \]
In both cases, we obtained a chain with three distinct
elements in $\mathcal{J}$, a contradiction to (ii). Thus, 
$\wp(U)^\UP$ is a Boolean lattice. As $\wp(U)^\UP$ is completely
distributive, it is atomic. Thus, $\mathrm{(ii)}\Rightarrow\mathrm{(iii)}$. 

\medskip\noindent%
$\mathrm{(iii)}\Rightarrow\mathrm{(ii)}$:
Let $\wp(U)^\UP$ be an atomic Boolean lattice. 
It is completely distributive and all its completely join-irreducible 
elements are atoms. 

Because $\wp(U)^\UP$ is a Boolean lattice, it is self-dual, 
that is, $(\wp(U)^\UP,\subseteq) \cong (\wp(U)^\UP,\supseteq)$.
Since $(\wp(U)^\UP,\supseteq)$ is isomorphic to 
$(\wp(U)^{\Up}, \subseteq)$, $\wp(U)^{\Up}$ is also an 
atomistic Boolean lattice such that all its completely 
join-irreducible elements are atoms. 

The facts that $\wp(U)^\UP$ and $\wp(U)^\Up$ are completely
distributive and spatial imply that $\mathrm{DM(RS)}$ is completely distributive and spatial in view of 
Propositions~{4.1} and~{4.4} of \cite{JarRad25}.

Let us observe first that for any $j \in \mathcal{J}^{\circ}$ 
and $k \in \mathcal{J}$, neither $j < k$ nor $k < j$ is  
possible. Since $k < j$ implies $j = g(j) < g(k)$ and 
$g(k) \in \mathcal{J}$, it is enough to consider only
the first case. Then $g(k)<g(j)=j<k$ implies $k \in \mathcal{J}^{+}$. 
Now $j \in \mathcal{J}^{\circ}$ means that 
$j = (\{x\},\{x\}^\UP)$ for some $x\in \mathcal{S}$. Similarly,
$k \in \mathcal{J}^{+}$ gives that 
$k=(\{y\}^{\Up \DOWN},\{y\}^{\Up\UP})$ for some
$\{y\}^{\Up}\in\mathcal{J}(\wp(U)^\Up)$ and $y\notin \mathcal{S}$. 
Since $(\{x\},\{x\}^\UP) < (\{y\}^{\Up\DOWN},\{y\}^{\Up\UP})$,
we have $\{x\} \subseteq \{y\}^{\Up\DOWN}$ and
$\{x\}^\Up \subseteq \{y\}^{\Up\DOWN \Up} = \{y\}^\Up$.
Because $\{y\}^\Up$ is a completely join-irreducible element of 
$\wp(U)^\Up$, it is an atom, as we noted above. 
We get $y \in \{y\}^\Up = \{x\}^\Up = R(x) = \{x\}$. 
Thus, $y = x \in \mathcal{S}$, a contradiction.

Assume for the sake of contradiction that $\mathcal{J}$ 
contains a chain $j < k < p$.
As we already noted, none of these elements can belong to 
$\mathcal{J}^{\circ}$. 
Then either: (a) at least two elements of this chain have the 
form $(\emptyset,\{x\}^\UP),(\emptyset,\{y\}^\UP)\in \mathcal{J}^-$ with 
$\{x\}^\UP$ and $\{y\}^\UP$ belonging to $\mathcal{J}(\wp(U)^\UP)$, 
$x,y\in \mathcal{S}$, and $(\emptyset,\{x\}^\UP)<(\emptyset,\{y\}^\UP)$; 
or (b) at least two elements of this chain have the form 
$(\{x\}^{\Up\DOWN},\{x\}^{\Up\UP})$, $(\{y\}^{\Up\DOWN},\{y\}^{\Up\UP})\in \mathcal{J}^{+}$ with 
$\{x\}^{\Up}, \{y\}^{\Up}\in\mathcal{J}(\wp(U)^{\Up})$, 
$x,y \notin \mathcal{S}$, and 
$(\{x\}^{\Up\DOWN},\{x\}^{\Up\UP})<(\{y\}^{\Up\DOWN},\{y\}^{\Up\UP})$.

In case (a), $(\emptyset,\{x\}^\UP)<(\emptyset,\{y\}^\UP)$ implies 
$\{x\}^\UP \subseteq\{y\}^\UP$, $\{x\}^\UP\neq\{y\}^\UP$.
However, this is impossible, because $\{x\}^\UP$ and
$\{y\}^\UP$ are atoms of the Boolean lattice $\wp(U)^\UP$. 

In case (b), $\{x\}^{\Up\DOWN} \subseteq \{y\}^{\Up\DOWN}$ 
implies $\{x\}^{\Up} = \{x\}^{\Up\DOWN\Up} \subseteq 
\{y\}^{\Up \DOWN\Up} = \{y\}^{\Up}$. As $\{x\}^{\Up}$ and $\{y\}^{\Up}$ 
are atoms of the Boolean lattice $\wp(U)^\Up$, 
we get $\{x\}^{\Up} = \{y\}^{\Up}$, which implies 
$j=(\{x\}^{\Up\DOWN },\{x\}^{\Up\UP}) = 
(\{y\}^{\Up\DOWN},\{y\}^{\Up\UP})=k$, a contradiction again. Thus, 
$\mathcal{J}$ has at most two levels and 
$\mathrm{(iii)}\Rightarrow\mathrm{(ii)}$ holds.
\end{proof}

\begin{remark} \label{rem:Australian} 
A collection $\mathcal{H}$ of nonempty subsets of $U$ is called a 
\emph{covering} of $U$ if $\bigcup \mathcal{H} = U$. 
A covering $\mathcal{H}$ is \emph{irredundant} if $\mathcal{H} \setminus \{X\}$ 
is not a  covering for any $X \in \mathcal{H}$. 
A \emph{tolerance relation} is a reflexive and symmetric binary relation.
Each covering $\mathcal{H}$ \emph{induces} a tolerance relation
$\bigcup \{X \times X \mid X \in \mathcal{H}\}$.

In \cite{JarRad14}, we studied approximations and rough sets 
induced by a tolerance relation $R$. We showed that the complete 
lattices $\wp(U)^\DOWN$ and $\wp(U)^\UP$ are completely 
distributive if and only if $R$ is induced by an irredundant covering of $U$. Moreover, 
when $R$ is a tolerance induced by an irredundant covering of $U$, $\wp(U)^\DOWN$ and $\wp(U)^\UP$ 
are atomistic Boolean lattices where 
 $\{R(x)^\DOWN \mid \text{$R(x)$ is a block} \}$ and 
$\{R(x) \mid \text{$R(x)$ is a block} \}$ are the respective 
sets of atoms. A \emph{block} is a maximal $R$-compatible set, 
that is, every pair of its elements is $R$-related.

Furthermore, we demonstrated in \cite{JarRad14} that 
$\mathrm{RS}$ is a spatial and completely distributive 
lattice if and only if 
$R$ is induced by an irredundant covering. We also 
showed that whenever $R$ is a tolerance induced by an irredundant covering,  
$\mathrm{RS}$ is a Kleene algebra.

Theorem~\ref{thm:Stone_vs_Boolean} 
is a strong extension of our former results. Now we know 
that $\mathrm{DM(RS)}$ is a spatial, completely distributive, and
regular pseudocomplemented Kleene algebra whenever
$\wp(U)^\UP$ is an atomistic Boolean lattice, and this
result applies to any reflexive relation. For instance,
the relation $R$ considered in Section~\ref{sec:intro} is such that
$\wp(U)^\UP$ is finite, and hence atomistic, Boolean
lattice: $\{ \emptyset, \{1,2\}, \{1,3\}, U\}$. By
Theorem~\ref{thm:Stone_vs_Boolean}, 
$\mathrm{DM(RS)} = \mathrm{RS}$ defined by this 
$R$ is a regular double Stone algebra, as
mentioned in the Introduction.

Note also that if $R(x)$ is a block, then 
$R(x) \subseteq R(y)$ for all $y \in R(x)$ 
due to the maximality and compatibility of blocks.
Hence, $\mathfrak{core}R(x) = R(x)$.
\end{remark}

\section{Stone algebra} \label{sec:Stone}

A \emph{Stone algebra} is a pseudocomplemented distributive lattice $L$ such that for every $x \in L$,
\begin{equation}\label{eq:Stone_equation}
 x^* \vee x^{**} = 1.
\end{equation}
In a Stone algebra, the identity $(a\wedge b)^* = a^* \vee b^*$  holds.
In this section, we give necessary and sufficient conditions under which $\mathrm{DM(RS)}$ 
forms a completely distributive and spatial Stone algebra. 
We begin with the following observation.

\begin{lemma} \label{lem:Stone}
Let $L$ be a completely distributive and spatial
lattice. Then the following are equivalent.
\begin{enumerate}[label={\rm (\roman*)}]
\item The $p$-algebra $(L,\vee,\wedge,^*,0,1)$ is a Stone algebra.
\item For any completely join-irreducible element $j$ and all
elements $x,y \in L$ with $0 < x,y \leq j$, we have $x \wedge y \neq 0$.
\end{enumerate}
\end{lemma}

\begin{proof}
$\mathrm{(i)}\Rightarrow\mathrm{(ii)}$:  
Suppose that $(L,\vee,\wedge,^*,0,1)$ is a Stone algebra.
Let $j \in \mathcal{J}$ and $x,y\in L$ such that $0<x,y\leq j$.
Assume for the sake of contradiction that $x \wedge y = 0$, that is, 
$y\leq x^*$. Since $L$ is a Stone algebra,
$x^*\vee x^{**}=1$ implies $j\leq x^*\vee x^{**}$.
Because any completely join-irreducible element of a completely distributive lattice is completely join-prime, it follows that
(a) $j \leq x^*$ or (b) $j\leq x^{**}$. Case~{(a)}: We have  
$x = x \wedge j \leq x \wedge x^* = 0$, a contradiction. 
Case~{(b)}: Now $y = y \wedge j \leq x^* \wedge x^{**}=0$, which is also a contradiction. Thus, we conclude that $x \wedge y\neq 0$.

$\mathrm{(ii)}\Rightarrow\mathrm{(i)}$:  
Assume for the sake of contradiction that (ii) holds but there exists 
$x\in L$ such that  $x^*\vee x^{**}<1$. Since $L$ is spatial,
$\bigvee \mathcal{J} = 1$. Hence, there exists  
$j \in \mathcal{J}$ with $j \nleq x^*\vee x^{**}$. 
Define $a := j \wedge x^*$ and $b := j\wedge x^{**}$. 
Note that $a \neq 0$, as $j \wedge x^* = 0$ would imply 
$j \leq x^{**}\leq x^*\vee x^{**}$, a contradiction.
Similarly, $b \neq 0$ since $j\wedge x^{**} = 0$ would
give $j \leq x^{***} = x^* \leq x^*\vee x^{**}$.
Thus, we have $0 < a,b \leq j$. However, 
$a \wedge b \leq x^*\wedge x^{**}=0$, which contradicts (ii). 
It follows that $x^* \vee x^{**}=1$ must hold for all $x\in L$. 
Therefore, $(L,\vee,\wedge,^*,0,1)$ is a Stone algebra.
\end{proof}

\begin{example}
(a) Let us consider three examples presented in \cite{Blyth2005}. 
Any bounded chain is a distributive pseudocomplemented lattice
such that $\mathcal{J}$ consists of all nonzero elements.
Obviously,  condition (ii) of Lemma~\ref{lem:Stone} is true for all elements $x,y \in \mathcal{J}$.
Thus, any bounded chain is a Stone algebra, as stated also in
Example~{7.6} of Blyth's book.

Example~{7.7} says that if $L_1,\dots,L_n$ are Stone algebras, 
then the direct product $\prod_{i=1}^n L_i$ is a Stone algebra 
such that the pseudocomplement $(x_1,\dots,x_n)^*$ is 
$(x_1^*, \ldots, x_n^*)$. It follows that
any finite product of bounded chains is a Stone algebra.

Let $L$ be the set of positive divisors of a positive integer 
$n$. We may order this set by divisibility: $a \le b \iff a \mid b$.
If $n = p^{\alpha}$,  where $p$ is a prime, then this lattice is an 
$(\alpha + 1)$-element chain.  
If $n = p_1^{\alpha_1} \cdot p_2^{\alpha_2} \cdots p_k^{\alpha_k}$, 
where $p_1, \ldots, p_k$ are distinct primes, then $L$
is a direct product of $k$ finite chains. By the above,
$L$ is a Stone algebra.

\medskip%
(b) Also, the finite distributive lattice $\mathrm{RS}$ on page~\pageref{fig:stonean} in Figure~\ref{fig:stonean} (right) 
 is a Stone algebra. Its completely join-irreducible elements 
are marked with black circles. It is easy to verify that condition (ii) of Lemma~\ref{lem:Stone} holds.
\end{example}

\begin{lemma} \label{lem:Stone_backwards}
Let $R$ be a reflexive relation on $U$. If $\mathrm{DM(RS)}$ is a Stone algebra, then $\wp(U)^\UP$ is a Stone algebra
such that $B^{* \Down}\cup B^{**\Down}=U$ for any $B\in\wp(U)^\UP$.
\end{lemma}

\begin{proof}
Let $B\in\wp(U)^\UP$. By Lemma~\ref{lem:exists_pair},
$(B^{\Down\DOWN},B)$ belongs to $\mathrm{DM(RS)}$. Since
$\mathrm{DM(RS)}$ is a Stone lattice,  
$(B^{\Down\DOWN},B)^*\vee (B^{\Down\DOWN},B)^{**}=(U,U)$.
By Proposition~\ref{prop:pseudomplement},
$(B^{\Down\DOWN},B)^* = (B^{*\Down\DOWN},B^*)$ and
$(B^{\Down\DOWN},B)^{**} = (B^{** \Down\DOWN},B^{**})$. Hence,
\begin{equation} \label{eq:Stone}
(B^{*\Down\DOWN},B^*)\vee (B^{** \Down\DOWN},B^{**})=(U,U).
\end{equation}
From the second components of \eqref{eq:Stone}, 
we get $B^*\vee B^{**} =U$, meaning that
$\wp(U)^\UP$ is a Stone algebra. 
The first components of \eqref{eq:Stone}
yield 
$(B^{*\Down \DOWN} \vee B^{**\Down\DOWN}) = 
(B^{*\Down \DOWN}\cup B^{**\Down\DOWN})^{\Up\DOWN} = U$.
This gives $(B^{*\Down \DOWN}\cup B^{**\Down\DOWN})^\Up=U$ and
further $B^{*\Down \DOWN \Up}\cup B^{**\Down\DOWN\Up} = U$.
Since $B^{*\Down\DOWN\Up} \subseteq B^{*\Down}$ and
$B^{**\Down \DOWN\Up}\subseteq B^{**\Down}$, we infer that
$B^{*\Down}\cup B^{**\Down}=U$. 
\end{proof}

\begin{corollary} \label{cor:Stone_nonempty}
Let $R$ be a reflexive relation. If $\mathrm{DM(RS)}$ is a Stone algebra, 
then for any $x,y,z\in U$ satisfying 
$\{x\}^\UP,\{y\}^\UP \subseteq \{z\}^\UP$, the meet
$\{x\}^\UP\wedge\{y\}^\UP$ is nonempty in $\wp(U)^\UP$.
\end{corollary}

\begin{proof} Assume $\{x\}^\UP,\{y\}^\UP \subseteq \{z\}^\UP$. 
Since  $\mathrm{DM(RS)}$ is a Stone algebra,
Lemma~\ref{lem:Stone_backwards} implies 
$\{y\}^{\UP * \Down} \cup\{y\}^{\UP **\Down} = U$.
It follows that $z\in\{z\}^\UP \subseteq 
\{y\}^{\UP * \Down}\cup\{y\}^{\UP**\Down}$, so
$z\in\{y\}^{\UP * \Down}$ or $z\in\{y\}^{\UP ** \Down}$. 
Suppose for the sake of contradiction that 
$\{x\}^\UP\wedge \{y\}^\UP=\emptyset$ in $\wp(U)^\UP$, which implies
$\{x\}^\UP \subseteq \{y\}^{\UP *}$. 
Now $z\in\{y\}^{\UP * \Down}$
implies $\{z\}^\UP\subseteq \{y\}^{\UP * \Down \UP} \subseteq
\{y\}^{\UP *}$, whence $y\in\{y\}^\UP = \{y\}^\UP\wedge \{z\}^\UP \subseteq 
\{y\}^\UP\wedge \{y\}^{\UP * }=\emptyset$, a contradiction. 
Similarly, $z\in\{y\}^{\UP**\Down}$ implies
$\{z\}^\UP\subseteq\{y\}^{\UP**}$, and this
yields $x\in\{x\}^\UP = \{x\}^\UP\wedge \{z\}^\UP \subseteq
\{y\}^{\UP *} \wedge \{y\}^{\UP**} = \emptyset$, a contradiction again. 
In either case, we reach a contradiction. Thus, we must have
$\{x\}^\UP\wedge\{y\}^\UP\neq\emptyset$.
\end{proof}

We introduce the following two conditions:
\begin{equation} \tag{St1} \label{eq:St1}
\begin{minipage}[c]{0.9\textwidth}
     If $\{x\}^\UP, \{y\}^\UP \subseteq \{p\}^\UP$, then there exists $z \in U$ such that $\{z\}^\UP \subseteq \{x\}^\UP, \{y\}^\UP$
\end{minipage}
\end{equation}
\begin{equation} \tag{St2} \label{eq:St2}
\begin{minipage}[c]{0.9\textwidth}
    If $\{x\}^\UP, \{y\}^\UP \subseteq \{p\}^{\Up \UP}$, where $\{p\}^\Up \in \mathcal{J}(\wp(U)^\Up)$, then
    there exists $z \in U$ such that $\{z\}^\UP \subseteq \{x\}^\UP, \{y\}^\UP$.
\end{minipage}
\end{equation}

\begin{theorem} \label{thm:Stone}
Let $R$ be a reflexive relation on $U$. If $\mathrm{DM(RS)}$ is a completely 
distributive and spatial lattice, then the following are equivalent.
\begin{enumerate}[label={\rm (\roman*)}]
\item Condition \eqref{eq:St1} holds.
\item Condition \eqref{eq:St2} holds.
\item $\mathrm{DM(RS)}$ is a Stone algebra.
\end{enumerate}
\end{theorem}

\begin{proof}
$\mathrm{(i)}\Rightarrow\mathrm{(ii)}$:  
Assume that \eqref{eq:St1} holds and let $x,y,p\in U$
such that $\{x\}^\UP,\{y\}^\UP \subseteq \{p\}^{\Up \UP}$ and 
$\{p\}^\Up \in\mathcal{J}(\wp(U)^\Up)$. 
Because $\mathrm{DM(RS)}$ is spatial, $\wp(U)^\UP$ is also spatial by Proposition~{4.4} of \cite{JarRad25}.
Hence, $\{x\}^\UP$ and $\{y\}^\UP$ are joins of some completely 
join-irreducible elements of $\wp(U)^\UP$. It follows that
there exist $\{u\}^\UP$, $\{v\}^\UP \in\mathcal{J}(\wp(U)^\UP)$ 
such that  $\{u\}^\UP \subseteq \{x\}^\UP$ and
$\{v\}^\UP \subseteq \{y\}^\UP$. From  (GC4), it follows that
\[  \{u\}^\UP,  \{v\}^\UP \subseteq \{p\}^{\Up \UP} =  \bigcup \{ \{a\}^\UP \mid a \in \{p\}^\Up \} . \]
The fact that $\mathrm{DM(RS)}$ is completely distributive implies that
$\wp(U)^\UP$ is completely distributive. In a completely distributive 
lattice, all completely join-irreducible elements are completely
join-prime. Therefore, $\{u\}^\UP$ and  $\{v\}^\UP$ are completely
join-prime. It follows that there exist $a_1, a_2 \in \{p\}^\Up = R(p)$
such that $\{u\}^\UP \subseteq \{a_1\}^\UP$
and  $\{v\}^\UP \subseteq \{a_2\}^\UP$.

Since $\{p\}^\Up$ belongs to $\mathcal{J}(\wp(U)^\Up)$, 
$\mathfrak{core}R(p) \neq \emptyset$. Hence, there exists 
$w\in\mathfrak{core}R(p)$. By \cite[Lemma~{4.13}]{JarRad25}, 
$p \in \mathfrak{core} \breve{R}(w)$. 
Since $a_1,a_2 \in R(p)$, $p \in \breve{R}(a_1)$ and 
$p \in \breve{R}(a_2)$. By the definition of 
$\mathfrak{core}\breve{R}(w)$, we obtain 
$\{w\}^\UP = \breve{R}(w) \subseteq \breve{R}(a_1) = \{a_1\}^\UP$ 
and $\{w\}^\UP = \breve{R}(w) \subseteq \breve{R}(a_2) = \{a_2\}^\UP$. 
Since $\{u\}^\UP, \{w\}^\UP \subseteq\{a_1\}^\UP$,  there is $z_1$ 
such that $\{z_1\}^\UP \subseteq \{u\}^\UP, \{w\}^\UP$
by \eqref{eq:St1}. Analogously, since $\{v\}^\UP,\{w\}^\UP \subseteq \{a_2\}^\UP$, 
there is $z_2$ with  $\{z_2\}^\UP \subseteq \{v\}^\UP, \{w\}^\UP$.
Because $\{z_{1}\}^\UP,\{z_{2}\}^\UP\subseteq \{w\}^\UP$, 
we can apply \eqref{eq:St1} again and obtain that
there is $z$ such that $\{z\}^\UP \subseteq
\{z_1\}^\UP,\{z_2\}^\UP$. Then, 
$\{z\}^\UP \subseteq\{z_1\}^\UP \subseteq \{u\}^\UP \subseteq\{x\}^\UP$ and 
$\{z\}^\UP \subseteq \{z_2\}^\UP \subseteq \{v\}^\UP \subseteq\{y\}^\UP$, 
proving \eqref{eq:St2}. 

\smallskip\noindent%
$\mathrm{(ii)}\Rightarrow\mathrm{(i)}$:  
Assume that \eqref{eq:St2} holds, and let $x,y,p\in U$ be such that
$\{x\}^\UP, \{y\}^\UP \subseteq \{p\}^\UP$. 
Since $\mathrm{DM(RS)}$ is spatial by assumption, 
in light of Proposition~{4.4} of \cite{JarRad25}, $\wp(U)^\Up$ is also spatial. 
Hence $U\in\wp(U)^\Up$ is equal to the join of completely 
join-irreducible elements of $\wp(U)^\Up$, that is, 
$U= \bigcup \{\{q\}^\Up \mid \{q\}^\Up \in\mathcal{J}(\wp(U)^\Up)\}$. 
This implies $\{p\}\subseteq \{q\}^\Up$ for some
$\{q\}^\Up \in \mathcal{J}(\wp(U)^\Up)$. It follows that
$\{x\}^\UP, \{y\}^\UP \subseteq \{p\}^\UP \subseteq
\{q\}^{\Up \UP}$. By \eqref{eq:St2} there exists
an element $z\in U$ such that $\{z\}^\UP \subseteq
\{x\}^\UP, \{y\}^\UP$, which completes the proof. 

\smallskip\noindent%
$\mathrm{(i),(ii)}\Rightarrow\mathrm{(iii)}$: 
Suppose that \eqref{eq:St1} and \eqref{eq:St2} hold.
We will prove that in this case the condition in 
Lemma~\ref{lem:Stone}(ii) is satisfied
in $\mathrm{DM(RS)}$ for an arbitrary completely join-irreducible element of $\mathrm{DM(RS)}$. There are two kinds of completely 
join-irreducible elements in $\mathrm{DM(RS)}$:
\begin{enumerate}
\item[(a)] $(\{z\}^\DOWN,\{z\}^\UP)$, where
$\{z\}^\UP$ is completely join-irreducible in 
$\wp(U)^{\UP}$ and $z\notin\mathcal{S}$;

\item[(b)] $(\{z\}^{\Up\DOWN},\{z\}^{\Up\UP})$, 
where $\{z\}^{\Up}$ is completely join-irreducible in
$\wp(U)^{\Up}$.
\end{enumerate}

We first consider case (a). 
Let $\{p\}^{\UP}\in\mathcal{J}(\wp(U)^{\UP})$ be such that $p\notin\mathcal{S}$, and assume that
$(A,B),(C,D)\leq(\{p\}^\DOWN,\{p\}^{\UP})$ for some
$(A,B),(C,D)\in\mathrm{DM(RS)}$. Then $B = X^{\UP}$ and
$D = Y^{\UP}$ for some $X,Y\subseteq U$. Hence, for any $x\in X$ and
$y\in Y$ we get $\{x\}^{\UP}\subseteq X^{\UP}
=B\subseteq\{p\}^{\UP}$ and $\{y\}^{\UP}\subseteq
Y^{\UP}= D\subseteq\{p\}^{\UP}$. By condition \eqref{eq:St1},
there exists $z\in U$ such that $\{z\}^{\UP}\subseteq
\{x\}^{\UP} \cap \{y\}^{\UP}$. 
Since $z \in \{z\}^{\UP\Down}$, it follows that
$z \in \{z\}^{\UP\Down} \subseteq \{x\}^{\UP\Down} \subseteq B^\Down$. Similarly,
$z \in \{z\}^{\UP\Down} \subseteq \{y\}^{\UP\Down} \subseteq D^\Down$.
Hence, $z \in B^{\Down}\cap D^{\Down} = 
(B\cap D)^\Down \subseteq (B\cap D)^{\Down\UP}
= B\wedge D$. From this, we obtain  
$(A,B)\wedge (C,D) \neq (\emptyset,\emptyset)$. 
Thus, in this case, condition (ii) of
Lemma~\ref{lem:Stone} holds.

Next, we consider (b). Let $(\{p\}^{\Up\DOWN},\{p\}^{\Up\UP})$ be a 
completely join-irreducible element of $\mathrm{DM(RS)}$ such
that $\{p\}^\Up\in\mathcal{J}(\wp(U)^\Up)$. Assume that 
$(A,B),(C,D)\leq(\{p\}^{\Up\DOWN},\{p\}^{\Up\UP})$
for some $(A,B),(C,D)\in\mathrm{DM(RS)}$.
Then $B = X^{\UP}$ and $D = Y^{\UP}$ for some $X,Y\subseteq U$. 
Hence, for any $x\in X$ and $y\in Y$,  
$\{x\}^\UP \subseteq X^\UP = B\subseteq\{p\}^{\Up\UP}$ and
$\{y\}^\UP \subseteq Y^\UP = D \subseteq \{p\}^{\Up\UP}$. 
By condition \eqref{eq:St2}, there exists $z\in U$ such that
$\{z\}^\UP \subseteq \{x\}^\UP,\{y\}^\UP$. 
As in case (a), this implies that 
$z \in B\wedge D$. Therefore, 
$(A,B)\wedge (C,D) \neq (\emptyset,\emptyset)$ and Lemma~\ref{lem:Stone}(ii) holds. 
Thus, we can apply Lemma~\ref{lem:Stone} to $\mathrm{DM(RS)}$
to obtain that it is a Stone algebra.

\smallskip\noindent%
$\mathrm{(iii)}\Rightarrow\mathrm{(i)}$: 
Assume that $\mathrm{DM(RS)}$ is a Stone
algebra. We will prove that \eqref{eq:St1} holds. 

Let $x,y,p$ be such that  
$\{x\}^\UP, \{y\}^\UP \subseteq \{p\}^\UP$.
By Corollary~\ref{cor:Stone_nonempty}, 
$\{x\}^\UP \wedge \{y\}^\UP \neq \emptyset$.
Since $\{x\}^\UP\wedge\{y\}^\UP\in\wp(U)^\UP$, 
we have  $\{x\}^\UP\wedge\{y\}^{\UP}=Z^\UP$ for some 
$\emptyset \neq Z^\UP\in\wp(U)^\UP$. 
Therefore, there exists $z\in Z$ such that
$\{z\}^\UP\subseteq Z^\UP\subseteq \{x\}^\UP,\{y\}^\UP$. 
\end{proof}

\begin{remark} \label{rem:finite_stone}
If $\wp (U)^\UP$ is an atomic lattice,
then \eqref{eq:St1} is equivalent to: 
\begin{equation} \tag{$\text{St1}^\circ$} \label{eq:St1'}
\begin{minipage}[c]{0.9\textwidth}
For any $p \in U$, $\{p\}^\UP $ includes exactly one atom
of $\wp(U)^\UP$.
\end{minipage}
\end{equation}
Similarly, \eqref{eq:St2} is equivalent to:
\begin{equation} \tag{$\text{St2}^\circ$} \label{eq:St2'}
\begin{minipage}[c]{0.9\textwidth}
    For any $\{p\}^\Up \in \mathcal{J}(\wp(U)^\Up)$, $\{p\}^{\Up \UP}$ includes exactly one atom of $\wp(U)^\UP$.
\end{minipage}
\end{equation}
Indeed, suppose that $\wp(U)^\UP$ is atomic and that 
there are two atoms $\{x\}^\UP$ and $\{y\}^\UP$ below
$\{p\}^\UP$.
By \eqref{eq:St1}, there exists $\{z\}^\UP \subseteq
\{x\}^\UP, \{y\}^\UP$. Since $\wp(U)^\UP$ is atomic,
there is an atom $\{a\}^\UP$ included in $\{z\}^\UP$. 
This means that $\{a\}^\UP \subseteq \{x\}^\UP, \{y\}^\UP$.
Because these three are atoms, they must be equal. Thus,
$\{p\}^\UP$ contains at most one atom, and since $\wp(U)^\UP$ is 
atomic, it contains exactly one atom. Hence \eqref{eq:St1'} holds.

Conversely, if \eqref{eq:St1'} holds and $\{a\}^\UP$ is the only atom included in $\{p\}^\UP$, 
then this yields $\{a\}^\UP \subseteq \{x\}^\UP,\{y\}^\UP$ for any 
$\{x\}^\UP,\{y\}^\UP \subseteq  \{p\}^\UP$.
Thus, \eqref{eq:St1} and \eqref{eq:St1'} are equivalent.
The equivalence of \eqref{eq:St2} and \eqref{eq:St2'} can be 
shown analogously.

Therefore, in the case where $\mathrm{DM(RS)}$ is a completely 
distributive and spatial lattice and $\wp(U)^\UP$ is atomic, 
all four aforementioned conditions are equivalent according to
Theorem~\ref{thm:Stone}. This occurs, for instance, 
whenever $\mathrm{DM(RS)}$  is a finite distributive lattice: 
any finite lattice is atomic and spatial, and furthermore, any
finite distributive lattice is completely distributive. Thus, in
such a case, the fact that each $\{p\}^\UP$ contains exactly one
atom of $\wp(U)^\UP$ forces $\mathrm{DM(RS)}$ to be a Stone algebra.
\end{remark}

\begin{example}
Let $R$ be a reflexive relation on $U = \{1,2,3,4\}$ such that
\[ R(1) = \{1, 2, 3, 4\}, \quad
R(2) = \{2, 3\}, \quad
R(3) = \{2,3,4\}, \quad
R(4) = \{3,4\}.
\]
The relation is not symmetric because $(1,2) \in R$, but
$(2,1) \notin R$, for instance. 
The relation is also not transitive, because $(2,3) \in R$ and
$(3,4) \in R$, but $(2,4) \notin R$. The lattice $\wp(U)^\UP$
is depicted in Figure~\ref{fig:stonean} (left). 
Clearly, $\wp(U)^\UP$ is finite and distributive.
Since $\wp(U)^\UP$ has only one atom $\{1\}$, it follows that
$\mathrm{RS} = \mathrm{DM(RS)}$ is a Stone algebra by 
Remark~\ref{rem:finite_stone}. Its Hasse diagram is
in Figure~\ref{fig:stonean} (right). Its (completely)
join-irreducible elements are denoted by a filled circle.
\begin{figure}[htbp] 
\centering 
\includegraphics[width=0.6\textwidth]{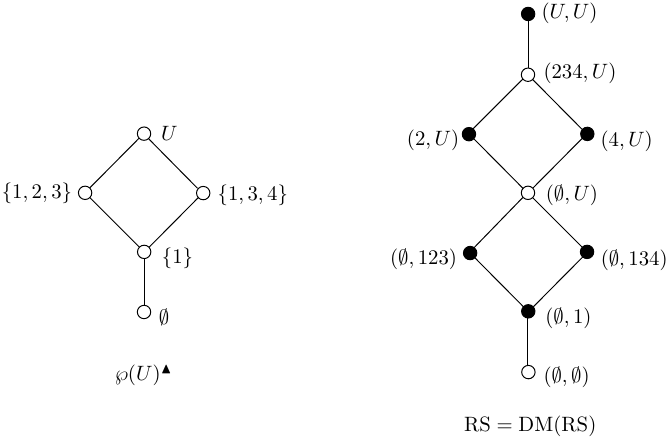} 
\caption{The Hasse diagrams of $\wp(U)^\UP$ and 
$\mathrm{RS} = \mathrm{DM(RS)}$.}
\label{fig:stonean}
\end{figure}

According to \eqref{eq:atoms},
the set of atoms of $\mathrm{RS}$ is
\[
\{ (\{x\}^\DOWN,\{x\}^\UP ) \mid 
\text{$\{x\}^\UP$ is an atom of $\wp(U)^\UP$} \}.
\]
Thus, $\mathrm{RS}$ also has only one atom which is
$(\{1\}^\DOWN, \{1\}^\UP) =  (\emptyset,\{1\})$.
This implies that for each element 
$(\emptyset,\emptyset) \neq \rho \in \mathrm{RS}$, we have
$\rho^* = (\emptyset,\emptyset)$ and $\rho^{**} = (U,U)$.
For these elements $\rho$, the Stone condition 
$\rho^* \vee \rho^{**} = (U,U)$ holds. 
Since $(\emptyset,\emptyset)^* = (U,U)$,  
$(\emptyset,\emptyset)$ also satisfies \eqref{eq:Stone_equation}.

As a pseudocomplemented Kleene algebra, $\mathrm{RS}$ is
not regular. This is evident because, as previously noted, all elements of $\mathrm{RS}$ other than $(\emptyset,\emptyset)$
share the same pseudocomplement $(\emptyset,\emptyset)$. 
Analogously, all elements different
from $(U,U)$ have the same dual pseudocomplement $(U,U)$. 
Therefore, the determination condition \eqref{eq:condM} 
does not hold. 
This is also clear by Theorem~\ref{thm:Stone_vs_Boolean},
as the set of completely join-irreducible elements of
$\mathrm{RS}$ has more than two levels.
\end{example}

We conclude this section by noting that in  \cite{JarRadVer09} 
we proved that for any quasiorder $R$, $\mathrm{RS}$ is a Stone lattice 
if and only if the relational product $\breve{R} \circ R$ is an equivalence. 
In fact,  it is not hard to verify that condition \eqref{eq:St1} is equivalent to this latter condition.

\section{Regular double Stone algebra}
\label{sec:regular_double_Stone}

A \emph{double Stone algebra} is a Stone algebra that is simultaneously a dual 
Stone algebra, that is, in addition to \eqref{eq:Stone_equation}, it satisfies the dual Stone identity:
\begin{equation} \label{eq:dual_Stone}
x^+ \wedge x^{++} = 0.
\end{equation}
A double Stone algebra is \emph{regular} if it satisfies condition (M).

To characterize the conditions under which $\mathrm{DM(RS)}$
forms a regular double Stone algebra, we consider a reflexive relation $R$ on $U$ satisfying:
\begin{equation} \tag{rSt} \label{eq:rSt}
\begin{minipage}[c]{0.9\textwidth}
 For any $x \in U$, $\{x\}^{\UP}$ is an atom in $\wp(U)^\UP$ and $\mathfrak{core}\breve{R}(x) \neq \emptyset$.
\end{minipage}
\end{equation}

\begin{lemma} \label{lem:rSt_implies}
Let $R$ be a reflexive relation $R$ on $U$.
If condition \eqref{eq:rSt} holds, then each
$\{x\}^\Up$ is a union of sets
$\{y\}^\Up$ such that $\mathfrak{core}R(y) \neq \emptyset$. 
\end{lemma}

\begin{proof} Assume \eqref{eq:rSt} holds. By Lemma~4.10 of \cite{JarRad25}, 
$\{x\}^\UP$ is completely join-prime 
whenever $\mathfrak{core}\breve{R}(x)$ is nonempty. 
Because for any $x \in U$, $\{x\}^\UP$ is completely join-prime,
every element of $\wp(U)^\UP$ is a union of some completely join-prime elements.
As noted in \cite{JarRad25}, this means that $\wp(U)^\UP$ 
is isomorphic to a complete field of sets.
Because $\wp(U)^\Up$ is dually isomorphic to $\wp(U)^\UP$, also
$\wp(U)^\Up$ is isomorphic to some complete field of sets. 
Hence, each element of $\wp(U)^\Up$ is a join of some completely 
join-prime elements of it. 
In view of \cite[Lemma~4.10]{JarRad25}, this implies that any 
$\{x\}^\Up$ is the union of some sets $\{y\}^\Up$ with 
$\mathfrak{core}R(y) \neq \emptyset$. 
\end{proof}

\begin{example} \label{ex:disprove}
Let us consider Example~\ref{exa:pseudocomplements}. Now
\[ 
\mathfrak{core}R(1) = \{2\}, \quad
\mathfrak{core}R(2) = \emptyset, \quad
\mathfrak{core}R(3) = \{3\}, \quad
\mathfrak{core}R(4) = \{4\}.
\]
Each $\{x\}^\Up = R(x)$ is a union of sets $\{y\}^\Up$ such that
$\mathfrak{core}R(y) \neq \emptyset$. However, \eqref{eq:rSt} does not hold.
The lattice $\wp(U)^\UP$ is depicted in Figure~\ref{fig:pseudocomplements} (left).
The sets $\{1\}^\UP = \breve{R}(1) = \{1,2,4\}$ and  
$\{3\}^\UP = \breve{R}(3) = \{2,3,4\}$ are not atoms.

More generally, let $U$ be finite and $\wp(U)^\UP$ be distributive.
Then, also $\wp(U)^\Up$ is distributive. Due to finiteness,
$\wp(U)^\Up$ is spatial. By Corollary~{4.12} of \cite{JarRad25}, the set of
completely join-prime elements of $\wp(U)^\Up$ consists of elements
$\{y\}^\Up = R(y)$ such that $\mathfrak{core}R(y) \neq \emptyset$.
Since $\wp(U)^\Up$ is distributive, its sets of completely join-irreducible
and join-prime elements coincide. Therefore, each $\{x\}^\Up$ is a union of sets
$\{y\}^\Up$ such that $\mathfrak{core}R(y) \neq \emptyset$.
However, \eqref{eq:rSt} does not hold because, by our next theorem, this would imply
that $\mathrm{DM(RS)}$ is a regular double Stone algebra. However, this is
clearly not true; see Figure~\ref{fig:pseudocomplements} (right),
for example.
\end{example}

\begin{theorem}
\label{thm:regular_double_stone}
Let $R$ be a reflexive relation on $U$. Then 
$\mathrm{DM(RS)}$  is a regular double Stone algebra
defined on a completely distributive and spatial lattice 
if and only if \eqref{eq:rSt} holds.
\end{theorem} 

\begin{proof} Assume that $\mathrm{DM(RS)}$ forms a regular double Stone 
algebra on a completely distributive spatial lattice. 
By Theorem~\ref{thm:Stone_vs_Boolean}, $\wp(U)^\UP$ is a Boolean lattice.
Because $\mathrm{DM(RS)}$ is completely distributive lattice, 
$\wp(U)^\UP$ is a completely distributive Boolean lattice.
This means that $\wp(U)^\UP$ is atomistic and 
its completely join-irreducible elements coincide with its atoms.
Therefore, each atom of $\wp(U)^\UP$ has the form
$\{a\}^{\UP}$ for some $a\in U$. 
Let $x \in U$.  Next we show that $\{x\}^\UP$ is an atom.
It is now clear that $\{x\}^{\UP}$ is the join of some atoms of 
$\wp(U)^\UP$ included in it. Clearly, at least one such an atom
must exist. Now let $\{a\}^{\UP}$ and  $\{b\}^{\UP}$ 
be two (not necessarily distinct) atoms of $\wp(U)^\UP$ with 
$\{a\}^{\UP}, \{b\}^{\UP} \subseteq \{x\}^{\UP}$. In view of
condition \eqref{eq:St1}, there exists an element 
$c\in U$ with  $\{c\}^{\UP}  \subseteq\{a\}^\UP,\{b\}^{\UP}$. Since
$\{a\}^{\UP}$ and $\{b\}^{\UP}$ are atoms, we get
$\{a\}^{\UP} = \{b\}^{\UP} = \{c\}^{\UP}$. This
means that $\{x\}^{\UP}$ contains only one atom
$\{a\}^{\UP}$ of $\wp(U)^\UP$. 
Because $\wp(U)^\UP$ is an atomistic lattice, this means
that $\{x\}^{\UP}$ itself must equal $\{a\}^{\UP}$. 
Furthermore, as any $\{x\}^{\UP}$ is an atom, it is join-irreducible.
Because $\wp(U)^\UP$ is completely distributive, $\{x\}^\UP$ is completely
join-prime. Hence, by Lemma~4.10 of \cite{JarRad25},
$\mathfrak{core}\breve{R}(x)\neq\emptyset$. Thus, 
\eqref{eq:rSt} holds. 

Conversely, assume that \eqref{eq:rSt} is satisfied. 
Then, $\{x\}^{\UP} = \breve{R}(x)$ is join-irreducible, 
and any $R(x)=\{x\}^{\Up}$ is
the union of some sets 
$\{p\}^{\Up}$ with $\mathfrak{core} R(p) \neq \emptyset$
by Lemma~\ref{lem:rSt_implies}.
Hence in view of \cite[Corollary 4.12]{JarRad25}, 
$\mathrm{DM(RS)}$ is a completely
distributive spatial lattice. Therefore, $\wp(U)^\UP$ is also spatial and
completely distributive. By \eqref{eq:rSt}, all completely 
join-irreducible elements of $\wp(U)^\UP$ are atoms, 
and hence $\wp(U)^\UP$ is a Boolean lattice. 
Therefore, using Theorem~\ref{thm:Stone_vs_Boolean} we have that $\mathrm{DM(RS)}$ is a 
regular double pseudocomplemented lattice.
Observe that condition \eqref{eq:St1} is also satisfied. 
Indeed, if $\{x\}^\UP, \{y\}^\UP \subseteq \{p\}^\UP$, the 
fact that all these are atoms implies that they are 
all equal. In view of Theorem~\ref{thm:Stone}, 
$\mathrm{DM(RS)}$ is a Stone
algebra. Because $\mathrm{DM(RS)}$ is a self-dual lattice, 
it is a double Stone algebra. Consequently, 
$\mathrm{DM(RS)}$ forms a regular double Stone algebra.
\end{proof}

It is known from the literature \cite{Comer93} that rough sets defined by an equivalence relation form a regular
double Stone algebra. Given an equivalence $E$ on $U$, the set $\wp(U)^\UP = \wp(U)^\DOWN$ forms an atomistic Boolean lattice in which each equivalence class $E(x) = \{x\}^\UP$ serves as an atom. Furthermore, the core $\mathfrak{core}\breve{E}(x) = \mathfrak{core}E(x) = E(x)$ is trivially nonempty for all $x \in U$, showing that condition \eqref{eq:rSt} holds.

\begin{example}\label{exa:demoing_rSt}
Let $R$ be a relation on $U = \{1,2,3,4\}$ such that
\[  R(1) = \{1,2,3,4\}, \ R(2) = \{1,2\}, \ R(3) = \{3\}, \ R(4) = \{4\}.\]
We have that $\wp(U)^\UP$ is a Boolean lattice consisting
of the sets $\emptyset$, $\{1,2\}$, $\{1,3\}$, $\{1,4\}$,
$\{1,2,3\}$, $\{1,2,4\}$, and $\{1,3,4\}$. Now
\[\{1\}^\UP =  
\{2\}^\UP = \{1,2\}, \ 
\{3\}^\UP = \{1,3\}, \ 
\{4\}^\UP = \{1,4\}.
\]
are the atoms of $\wp(U)^\UP$. Moreover, the cores
\[ 
\mathfrak{core}\breve{R}(1) = 
\mathfrak{core}\breve{R}(2) = \{2\}, \quad
\mathfrak{core}\breve{R}(3) = \{3\}, \quad
\mathfrak{core}\breve{R}(1) = \{4\}
\]
are all nonempty. Condition \eqref{eq:rSt} holds and by
Theorem~\ref{thm:regular_double_stone},
$\mathrm{DM(RS)}$ is a regular double Stone algebra.

The lattice
$\mathrm{DM(RS)}$ is depicted in Figure~\ref{fig:double_Stone}. 
\begin{figure}[htbp] 
\centering 
\includegraphics[width=60mm]{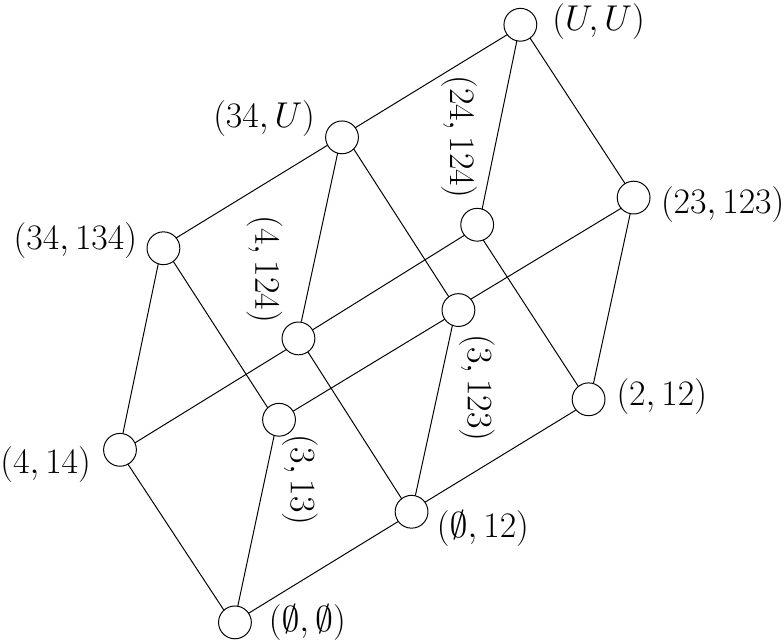} 
\caption{The Hasse diagram of $\mathrm{RS}$.} 
\label{fig:double_Stone} 
\end{figure}
The lattice $\mathrm{DM(RS)}$ is equal to $\mathrm{RS}$ and is
isomorphic to the product $
\mathbf{2} \times\mathbf{2} \times \mathbf{3}$, where 
$\mathbf{2}$ and $\mathbf{3}$ are chains of 
two and three elements, respectively.  
\end{example}

\begin{proposition} \label{prop:irredundant_cover}
Let $R$ be a reflexive relation on $U$. Then the following are equivalent:
\begin{enumerate}[label={\rm (\roman*)}]
\item $\{ \{x\}^\UP \mid x \in U\}$ is an irredundant covering of $U$.

\item Condition \eqref{eq:rSt} holds.
\end{enumerate}
\end{proposition}

\begin{proof}
$\mathrm{(i)}\Rightarrow\mathrm{(ii)}$: Let $x \in U$.
As $\{ \{x\}^\UP \mid x \in U\}$ forms an irredundant 
covering of $U$, the inclusion 
$\{y\}^\UP \subseteq\{x\}^\UP$ implies 
$\{y\}^\UP = \{x\}^\UP$  for any $y \in U$.
Let $\emptyset \neq Y^\UP \subseteq \{x\}^\UP$. Then $Y$ is nonempty
and it contains at least one element $y$.
Now $\{y\}^\UP \subseteq Y^\UP \subseteq \{x\}^\UP$ and 
$\{y\}^\UP = \{x\}^\UP$ imply $Y^\UP = \{x\}^\UP$. This means that $\{x\}^\UP$ is an atom of $\wp(U)^\UP$.

Next we prove that 
$\mathfrak{core}\breve{R}(x) \neq \emptyset$.
Because the covering 
$\{ \{a\}^\UP \mid a \in U\}$ is irredundant,
there exists an element $y$ in the difference
\[
\{x\}^\UP \setminus   
\bigcup \{\{b\}^\UP \mid \{b\}^\UP \neq\{x\}^\UP\}.
\]
Suppose that $y \in \breve{R}(c) = \{c\}^\UP$
for some $c \in U$. We must have 
$\breve{R}(c) = \{c\}^\UP = \{x\}^\UP = \breve{R}(x)$.
Thus, $\breve{R}(x) \subseteq \breve{R}(c)$ yields 
$y \in \mathfrak{core}\breve{R}(x)$.

\medskip\noindent%
$\mathrm{(ii)}\Rightarrow\mathrm{(i)}$: 
Suppose that \eqref{eq:rSt} holds and 
that the covering $\{ \{a\}^\UP \mid a \in U\}$ is 
not irredundant. This means that there exists $x \in U$
such that 
$\{x\}^\UP \subseteq \bigcup\{\{y\}^\UP \mid \{y\}^\UP \neq\{x\}^\UP\}$. 
Now $\emptyset\neq\mathfrak{core}\breve{R}(x) \subseteq\breve{R}(x)
= \{x\}^\UP$. Hence, for any $w \in \mathfrak{core}\breve{R}(x)$,
there is $y\in U$ with $\{y\}^\UP \neq \{x\}^\UP$
and $w\in \{y\}^\UP = \breve{R}(y)$. Then $\{x\}^\UP = \breve{R}(x) 
\subseteq \breve{R}(y) = \{y\}^\UP$. Since
$\{y\}^\UP$ is an atom in $\wp(U)^\UP$, we get
$\{y\}^\UP = \{x\}^\UP$, a contradiction. Thus,
we have an irredundant covering.
\end{proof}

\begin{definition} \label{def:clinker}
Let $R$ be a reflexive relation. We say that
$R$ is a \emph{clinker equivalence} if 
$\{ \{x\}^\UP \mid x \in U\}$ is an irredundant covering of $U$.
\end{definition}

\begin{remark}
The term clinker equivalence was chosen because the overlapping sets $\{x\}^\UP$ resemble the overlapping wooden strakes of a Viking ship's hull, referencing the historical clinker-built method of Norse shipbuilding.
\end{remark}

\begin{example}\label{exa:clinker_demo}
In this example, we informally introduce the characteristic properties of a clinker equivalence, which can be viewed as a specific extension of an equivalence relation.
\begin{figure}[htbp] 
\centering 
\includegraphics[width=80mm]{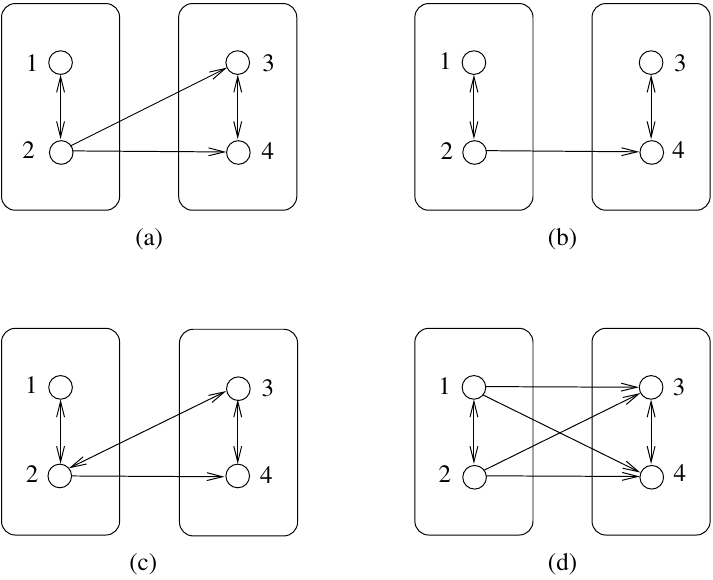} 
\caption{Examples and non-examples of clinker equivalences.} 
\label{fig:clinkers} 
\end{figure}

Let $U = \{1,2,3,4\}$. Consider the relation in case (a) of Figure~\ref{fig:clinkers}. The classes of the underlying equivalence are $\{1,2\}$ and $\{3,4\}$. There is an arrow from $2$ to the other class. If there is an arrow from an element in class $A$ to an element in another class $B$, then the arrow must go to all elements of $B$. In addition, if there is an arrow from class $A$ to some class $B$, there cannot be an arrow from class $B$ back to class $A$. Furthermore, there must be at least one element in each class which is not related to any element outside its class. In this case, $\{1\}^\UP = \{2\}^\UP = \{1,2\}$ and $\{3\}^\UP = \{4\}^\UP = \{2,3,4\}$ form an irredundant covering.

In case (b), there is an arrow from $2$ to $4$, but there should also be an arrow from $2$ to $3$. This relation is not a clinker equivalence, because $\{1\}^\UP = \{2\}^\UP = \{1,2\}$, $\{3\}^\UP = \{3,4\}$, and $\{4\}^\UP = \{2,3,4\}$ do not form an irredundant covering.

The relation in (c) is otherwise similar to (a), but now there is an arrow back from $3$ to $2$. This relation is not a clinker equivalence, because $\{1\}^\UP = \{1,2\}$, $\{2\}^\UP = \{1,2,3\}$, and $\{3\}^\UP = \{4\}^\UP = \{2, 3,4\}$ do not form an irredundant covering.

Finally, (d) is not a clinker equivalence, because all elements in the class $\{1,2\}$ are related outside their class. Now $\{1\}^\UP = \{2\}^\UP = \{1,2\}$ and $\{3\}^\UP = \{4\}^\UP = \{1, 2,3,4\}$ do not form an irredundant covering.

In this example, there are only two underlying equivalence classes. Naturally, there can be more than two classes. An in-depth characterisation of clinker equivalences will be presented in a paper currently in preparation.
\end{example}

By Proposition~\ref{prop:irredundant_cover}, $R$
is a clinker equivalence if and only if condition~\eqref{eq:rSt}
holds. Therefore, by applying Theorem~\ref{thm:regular_double_stone},
we obtain our following proposition.

\begin{proposition} \label{cor:clinker_characterization}
Let $R$ be a reflexive relation on $U$. Then
$\mathrm{DM(RS)}$ forms a regular double Stone algebra 
defined on a completely
distributive spatial lattice if and only if $R$ is a clinker
equivalence.
\end{proposition}

Our next example demonstrates the use of this proposition.

\begin{example}\label{exa:regular_double_stone}
Let us consider the relation $R$ of 
Example~\ref{exa:demoing_rSt}. Now,
\[\{1\}^\UP = \{2\}^\UP = \{1,2\}, \ 
\{3\}^\UP = \{1,3\}, \ 
\{4\}^\UP = \{1,4\}
\]
forms an irredundant covering of $U$. 
This means that $R$ is a clinker equivalence and
$\mathrm{DM(RS)}$ forms a regular double Stone algebra,
as we already noted in Example~\ref{exa:demoing_rSt}.
\end{example}

\begin{remark}
Proposition~\ref{prop:irredundant_cover} gives a simple method to 
generate regular double Stone algebras in terms of reflexive relations.
We first define an irredundant covering $\mathcal{C}$ of $U$. Then, we
attach to each element $x \in U$ a set of $B$ of $\mathcal{C}$ 
such that $x \in B$. Also, each set in $\mathcal{C}$ needs to be
attached to at least one element of $U$.
These sets will be the neighbourhoods of the inverse
relation of $R$, that is, $\breve{R}(x) = \{x\}^\UP$ for any $x \in U$.
Finally, construct the relation $R$ from $\breve{R}$.
\end{remark}

\section*{Concluding remarks and future work}

In this paper, we considered various algebras that can be defined on 
the completion $\mathrm{DM(RS)}$ of rough sets determined by a 
reflexive relation. In our previous work, we investigated a new way to obtain a Nelson algebra. Here, we have presented similar results for regular pseudocomplemented Kleene algebras and Stone algebras. The case of regular double Stone algebras is of special interest because it led us to introduce \emph{clinker equivalences}, which can be seen as a novel generalisation of equivalence relations. These reflexive relations need not be symmetric or transitive, but they define an irredundant covering of the underlying set that generalises partitions induced by equivalences. Our forthcoming paper will be devoted to an in-depth study and construction of clinker equivalences. In particular, we aim to deduce equivalent characterisations for clinker equivalences, clarify their relationship with usual equivalences, and demonstrate their applications in concept approximations, rough set constructions, and classification problems. Another interesting direction will be their generation by information systems and the identification of constructions that preserve them. We believe that these relations have numerous applications in theoretical computer science and clustering methods used in engineering sciences.

\section*{Acknowledgements}
We thank the reviewers for their careful reading and constructive feedback, which have significantly improved the clarity, structure, and overall quality of our manuscript.

\printbibliography

@book{Blyth2005,
  title={Lattices and Ordered Algebraic Structures},
  author={Blyth, T.~S.},
  year={2005},
  publisher={Springer, London},
  doi={10.1007/b139095}
}

@ARTICLE{GeWa92, 
   author = "Mai Gehrke and Elbert Walker",
   title =  "On the structure of rough sets", 
   journal =  "Bulletin of Polish Academy of Sciences, Mathematics", 
   volume =  "40", 
   pages =  "235-245", 
   year =  "1992"
}

@article{Jarvinen01,
author="J{\"a}rvinen, Jouni",
title="Approximations and rough sets based on tolerances",
journal = "Lecture Notes in Computer Science",
year="2001",
volume = 2005,
pages="182--189",
doi = "10.1007/3-540-45554-X_21"
}

@article{Jarv04,
author="J{\"a}rvinen, Jouni",
title="The ordered set of rough sets",
year="2004",
journal = "Lecture Notes in Computer Science",
volume = "3066",
pages="49--58",
doi = {10.1007/978-3-540-25929-9_5}
}

@article{Jarvinen2007,
    author={J\"arvinen, Jouni}, 
    title={Lattice theory for rough sets},
    journal = {Transaction on Rough Sets},
    pages={400-498},
    volume={VI},
    year={2007},  
    doi = {10.1007/978-3-540-71200-8_22}
}

@article{JarRadVer09,
  author    = {Jouni J{\"{a}}rvinen and S{\'{a}}ndor Radeleczki and Laura Veres},
  title     = {Rough sets determined by quasiorders},
  journal   = {Order},
  volume    = {26},
  pages     = {337--355},
  year      = {2009},
  doi       = {10.1007/s11083-009-9130-z}
}

@article{JarRad11,
    author={J{\"a}rvinen, Jouni and Radeleczki, S{\'a}ndor},
	Journal = {Algebra Universalis},
	Pages = {163--179},
	Title = {Representation of {N}elson algebras by rough sets determined by quasiorders},
	Volume = {66},
	Year = {2011},
	doi={10.1007/s00012-011-0149-9}}

@article{JarRad18,
  title={Representing regular pseudocomplemented {K}leene algebras by tolerance-based rough sets},
  author={J{\"a}rvinen, Jouni and Radeleczki, S{\'a}ndor},
  journal={Journal of the Australian Mathematical Society},
  volume={105},
  doi = {10.1017/S1446788717000283},
  pages={57--78},
  year={2018},
  publisher={Cambridge University Press}
}

@article{JarRad23,
author={J{\"a}rvinen, Jouni and Radeleczki, S{\'a}ndor},
title = {Pseudo-{K}leene algebras determined by rough sets},
journal = {International Journal of Approximate Reasoning},
volume = {161},
pages = {article ID 108991},
year = {2023},
doi = {10.1016/j.ijar.2023.108991},
}

@article{JarRad25,
author={J{\"a}rvinen, Jouni and Radeleczki, S{\'a}ndor},
title = {The structure of rough sets defined by reflexive relations},
journal = {International Journal of Approximate Reasoning},
volume = {185},
pages = {article ID 109471},
year = {2025},
doi = {10.1016/j.ijar.2025.109471}
}

@article{JarRad14,
author={J{\"a}rvinen, Jouni and Radeleczki, S{\'a}ndor},
title = {Rough sets determined by tolerances},
journal = {International Journal of Approximate Reasoning},
volume = {55},
pages = {1419-1438},
year = 2014,
doi = {10.1016/j.ijar.2013.12.005},
}

@article{Sankappanavar86,
author = {Sankappanavar, H. P.},
title = {Pseudocomplemented {O}ckham and {D}e~{M}organ Algebras},
journal = {Mathematical Logic Quarterly},
volume = {32},
pages = {385--394},
year = {1986},
doi = {10.1002/malq.19860322502}
}

@article{Umadevi2015,
  title={On the completion of rough sets system determined by arbitrary binary relations},
  author={Umadevi, Dwaraganathan},
  journal={Fundamenta Informaticae},
  volume={137},
  pages={413-424},
  year={2015},
  doi = {10.3233/FI-2015-1188}
}

@article{Var72,
	Author = {Jules Varlet},
	Journal = {Algebra Universalis},
    Volume = 2,
	Pages = {218--223},
	Title = {A regular variety of type 
    \textlangle 2, 2, 1, 1, 0, 0 \textrangle},
	Year = {1972},
    doi = {10.1007/BF02945029}
}

@book{Davey02,
	Author = {B.A. Davey and H.A. Priestley},
	Edition = {2nd edition},
	Publisher = {Cambridge University Press},
	Title = {Introduction to {L}attices and {O}rder},
	Year = {2002},
	doi = {10.1017/CBO9780511809088}
	}

@article{Erne_primer,
author = {Ern{\'e}, M. and Koslowski, J. and Melton, A. and 
          Stecker, G.~E.},
title = {A Primer on {G}alois connections},
journal = {Annals of the New York Academy of Sciences},
volume = {704},
pages = {103-125},
doi = {10.1111/j.1749-6632.1993.tb52513.x},
year = {1993}
}

@article{Balachandran55,
    author = {V.~K.~Balachandran},
    title = {On complete lattices and a problem of {B}irkhoff and {F}rink},
    journal = {Proceedings of the American Mathematical Society},
    year = 1955,
    volume=6,
    pages={548-553},
    doi = {10.2307/2033427}
}

@incollection{Comer93,
	Author = {S.~D. Comer},
	Booktitle = {Algebraic {M}ethods in {L}ogic and {C}omputer {S}cience},
	Editor = {C. Rauszer},
	Pages = {117--124},
	Publisher = {Institute of Mathematics, Polish Academy of Sciences},
	Series = {Banach Centre Publications},
	Title = {On connections between information systems, rough sets, and algebraic logic},
	Volume = {28},
	doi =  "10.4064/-28-1-117-124", 
	Year = 1993}

@article{Pawl82,
	Author = {Zdzis{\l}aw Pawlak },
	Journal = {International Journal of Computer \& 
    Information Sciences},
	Pages = {341--356},
	Title = {Rough sets},
	Volume = {11},
	Year = {1982},
    doi = {10.1007/BF01001956 }
}

@Article{PomPom88,
	author = "Jacek Pomyka{\l}a and Janusz A. Pomyka{\l}a",
	title = "The {S}tone algebra of rough sets",
	journal = "Bulletin of Polish Academy of Sciences. Mathematics",
	volume = 36,
	pages = "495-512",
	year = 1988
}

@article{Yao96,
author = {Y.Y. Yao and T.Y. Lin},
title = {Generalization of rough sets using modal logics},
journal = {Intelligent Automation \& Soft Computing},
volume = {2},
pages = {103--119},
year = {1996},
doi = {10.1080/10798587.1996.10750660}
}

@article{Tversky77,
    author =  "Tversky, Amos", 
    title = "Features of similarity",
    journal = "Psychological Review",
    volume = "84",
    year = 1977,
    pages = "327-352",
    doi = "10.1037/0033-295x.84.4.327"
}

@article{Kortelainen1994,
author = {Jari Kortelainen},
title = {On relationship between modified sets, topological spaces and rough sets},
journal = {Fuzzy Sets and Systems},
volume = {61},
pages = {91-95},
year = {1994},
doi = {10.1016/0165-0114(94)90288-7},
}

@inproceedings{Ganter07,
  author    = {Bernhard Ganter},
  title     = {Non-symmetric indiscernibility},
  booktitle = {Knowledge Processing and Data Analysis},
  series    = {Lecture Notes in Computer Science},
  volume    = {6581},
  pages     = {26--34},
  publisher = {Springer},
  year      = {2007},
  doi       = {10.1007/978-3-642-22140-8_2},
}

@Article{KumBan2015,
	author = "Kumar, Arun and Banerjee, Mohua",
	title = "Algebras of definable and rough sets in quasi order-based approximation spaces",
	journal = "Fundamenta Informaticae",
	volume = 141,
	pages = "37-55",
	year = 2015,
	doi = {10.3233/FI-2015-1262}
}

@article{Umadevi13,
  author    = {E. K. R. Nagarajan and D. Umadevi},
  title     = {Algebra of rough sets based on quasi order},
  journal   = {Fundamenta Informaticae},
  volume    = {126},
  pages     = {83--101},
  year      = {2013},
  doi       = {10.3233/FI-2013-872}
}

@INPROCEEDINGS{Qiao2012,
  author={Qiao, Quanxi},
  booktitle={2012 5th International Conference on BioMedical Engineering and Informatics}, 
  title={Topological structure of rough sets in reflexive and transitive relations}, 
  year={2012},
  pages={1585-1589},
  doi={10.1109/BMEI.2012.6513083}}

@article{JarvRadeRivi24,
author = {Jouni J{\"a}rvinen and S{\'a}ndor Radeleczki and Umberto Rivieccio},
title = {Nelson algebras, residuated lattices and rough sets: A survey},
journal = {Journal of Applied Non-Classical Logics},
volume = {34},
pages = {368--428},
year = {2024},
doi = {10.1080/11663081.2024.2336386}
}

@article{Ziarko1993,
author = {Wojciech Ziarko},
title = {Variable precision rough set model},
journal = {Journal of Computer and System Sciences},
volume = {46},
pages = {39-59},
year = {1993},
doi = {10.1016/0022-0000(93)90048-2},
}

@article{OrlPaw84,
author = {Ewa Or{\l}owska and Zdzis{\l}aw Pawlak},
title = {Representation of nondeterministic information},
journal = {Theoretical Computer Science},
volume = {29},
pages = {27-39},
year = {1984},
doi = {10.1016/0304-3975(84)90010-0},
}

@incollection{Pagl97,
	Address = {Berlin},
	Author = {P. Pagliani},
	Booktitle = {{I}ncomplete {I}nformation: {R}ough {S}et {A}nalysis},
	Editor = {E. Orlowska},
	Pages = {109--190},
	Publisher = {Physica-Verlag},
	Title = {Rough set systems and logico-algebraic structures},
	Year = {1997},
    doi = {10.1007/978-3-7908-1888-8_6}
    }

@article{SkowStep96,
author = {Skowron, Andrzej and Stepaniuk, Jaroslaw},
title = {Tolerance Approximation Spaces},
year = {1996},
volume = {27},
journal = {Fundamenta Informaticae},
pages = {245–253},
doi = {10.3233/FI-1996-27231},
}

@article{SlowVand00,
  author={Slowinski, R. and Vanderpooten, D.},
  journal={IEEE Transactions on Knowledge and Data Engineering}, 
  title={A generalized definition of rough approximations based on similarity}, 
  year={2000},
  volume={12},
  pages={331-336},
  doi={10.1109/69.842271}}

@article{SyauJia12,
    author = "Yu-Ru Syau and Lixing Jia",
    title = "Generalized rough sets based on reflexive relations",
    journal = "Communications in Information and Systems",
    volume = 12,
    year = 2012,
    pages = "233–249",
}

@misc{ManRad20,
      title={Algebraic approach to directed rough sets}, 
      author={Mani, A. and Radeleczki, S{\'a}ndor},
      year={2020},
      eprint={2004.12171},
      archivePrefix={arXiv},
      primaryClass={cs.LO},
      url={https://arxiv.org/abs/2004.12171}, 
}
\end{document}